\documentclass[12pt, reqno]{amsart}
\usepackage[utf8]{inputenc}
\usepackage{amsmath}
\usepackage{amsfonts}
\usepackage{amssymb}
\usepackage{amsthm}
\usepackage{bm}
\usepackage{bbm}
\usepackage{xcolor}
\theoremstyle{plain}
\newtheorem{theorem}{Theorem}[section]
\newtheorem{lemma}[theorem]{Lemma}

\newtheorem{proposition}[theorem]{Proposition}
\theoremstyle{definition}

\theoremstyle{remark}
\newtheorem{remark}[theorem]{Remark}

\title{Power savings for counting solutions to polynomial-factorial equations}
\author{Hung M. Bui, Kyle Pratt and Alexandru Zaharescu}
\address{Department of Mathematics, University of Manchester, Manchester M13 9PL, UK}
\email{hung.bui@manchester.ac.uk}
\address{All Souls College, Oxford OX1 4AL, UK}
\email{kyle.pratt@all-souls.ox.ac.uk}
\address{Department of Mathematics, University of Illinois at Urbana-Champaign, 1409 West Green Street, Urbana, IL 61801, USA
and Simion Stoilow Institute of Mathematics of the Romanian Academy, P.O. Box 1-764, RO-014700 Bucharest,
Romania}
\email{zaharesc@illinois.edu}

\subjclass[2010]{11D45}
\keywords{polynomial-factorial equation, simultaneous rational approximation, Pad\'e approximation}

\allowdisplaybreaks

\begin{document}
\date{}

\maketitle

\begin{abstract}
Let $P$ be a polynomial with integer coefficients and degree at least two. We prove an upper bound on the number of integer solutions $n\leq N$ to $n! = P(x)$ which yields a power saving over the trivial bound. In particular, this applies to a century-old problem of Brocard and Ramanujan. The previous best result was that the number of solutions is $o(N)$. The proof uses techniques of Diophantine and Pad\'e approximation.
\end{abstract}

\section{Introduction}

Henri Brocard posed, first in 1876\footnote{``Pour quelles valeurs du nombre entier $x$ l'expression $1.2.3.4\ldots x + 1$ est-elle un carr\'e parfait?''} \cite{Bro1876} and again in 1885\footnote{``Pour quelles valeurs de $x$ l'expression $1.2.3.4\ldots x + 1$ est-elle un carr\'e parfait?''} \cite{Bro1885}, the problem of finding all integer solutions to the equation
\begin{align}\label{eq:Brocard Ramanujan equation}
n! + 1 = x^2.
\end{align}
In 1913\footnote{``The number $1+n!$ is a perfect square for the values $4,5,7$ of $n$. Find other values.''} Ramanujan \cite{Ram1913} (see also \cite{BCK2001}) independently proposed the same problem. Computer calculations show that any solution $n$ other than the known solutions $n=4,5,7$ must be large \cite{BG2000}, and it is conjectured that no other solutions exist. While not offering a complete resolution of the problem, a weakened version of ABC conjecture (the ``weak form of Szpiro's conjecture'') implies that \eqref{eq:Brocard Ramanujan equation} has only finitely many solutions \cite{Over1993}.

More generally, one may inquire about solutions to the polynomial-factorial equation
\begin{align}\label{eq:polynomial factorial equation}
n! = P(x),
\end{align}
where $P$ is an integer polynomial of degree at least two. The ABC conjecture again implies there are only finitely many solutions (see \cite{Luca2002, Dab1996}), and the question is what can be said about \eqref{eq:polynomial factorial equation} unconditionally. For some polynomials one can prove that \eqref{eq:polynomial factorial equation} in fact has only finitely many solutions. For instance, $n!$ is never a perfect power for $n>1$, so the equation $n!=x^d$ has only one solution for any $d\geq 2$. Less trivially, it is known that $n! = x^d-1$ has no solutions if $d\geq 3$ \cite{EO1937, PS1973}. Berend and Harmse \cite{BH2006} made a general study of polynomial-factorial equations and proved several sufficient conditions for \eqref{eq:polynomial factorial equation} to have finitely many solutions. Takeda \cite{Tak2021} recently showed finiteness results for polynomials $P$ related to norm forms of algebraic number fields, and for $n$ restricted to certain infinite subsets of the integers. It is also possible to generalize \eqref{eq:polynomial factorial equation} even further and prove finiteness results for solutions to $f(P(x)) = n!$, where $f$ is an arithmetic function (see e.g. \cite{Sau2020}).

For a polynomial $P$ for which it is unknown at present whether \eqref{eq:polynomial factorial equation} has finitely many solutions, such as in the case of the Brocard-Ramanujan problem, one can at least ask for an upper bound on the number of solutions $n\leq N$ as $N \rightarrow \infty$. (Bounds for such exceptional sets have been proved in somewhat analogous situations e.g. \cite{Luca2007, LSS2014}.) Berend and Osgood \cite{BO1992} showed that \eqref{eq:polynomial factorial equation} has $o(N)$ solutions $n\leq N$, which answered a question of Erd\H{o}s. We improve on this result of Berend and Osgood, obtaining a power saving bound for the number of solutions of a polynomial-factorial equation.

\begin{theorem}[Power saving for number of solutions]\label{thm:main thm}
Let $P\in \mathbb{Z}[x]$ be a polynomial of degree $r\geq 2$, and let $s \in \mathbb{Z}\backslash \{0\}$ be fixed. Then there exists a positive constant $C=C(P,s)$, depending only on $P$ and $s$, such that for all positive integers $N$ we have
\begin{align*}
\#\{N \leq n < 2N : s \cdot n! = P(x) \text{ for some } x \in \mathbb{Z}\} \leq C N^{33/34}.
\end{align*}
\end{theorem}

In particular, Theorem \ref{thm:main thm} implies that the Brocard-Ramanujan equation \eqref{eq:Brocard Ramanujan equation} has $\ll N^{33/34}$ solutions with $n\leq N$.

\begin{remark}
We assume in the course of proving Theorem \ref{thm:main thm} that $N$ is sufficiently large depending on $P$ and $s$, since for bounded $N$ the statement of Theorem \ref{thm:main thm} follows trivially from adjusting the constant $C$.
\end{remark}

\begin{remark}
There is no harm in supposing in Theorem \ref{thm:main thm} that $s, x$, and the leading coefficient of $P$ are positive. By changing $x \rightarrow -x$ we may assume that $x$ is nonnegative, at the cost of adjusting the constant $C$, and possibly changing the sign of the leading coefficient of $P$ (if $P$ has odd degree). If $s$ and the leading coefficient of $P$ have opposite signs then there are at most finitely many solutions to $sn! = P(x)$. Therefore, we may assume $s$ and the leading coefficient of $P$ have the same sign, and hence assume both are positive upon multiplying through by $-1$ as necessary.
\end{remark}

\begin{remark}
The exponent $\frac{33}{34} = 0.97058\ldots$ is only an approximation to the best exponent obtained via our method, which is $12\sqrt{2}-16 + \epsilon = 0.97056\ldots$ for any small, fixed $\epsilon > 0$, provided $N$ is sufficiently large in terms of $\epsilon$ (see Proposition \ref{prop:not many close solutions} below).
\end{remark}

We give here an outline of the proof of Theorem \ref{thm:main thm}. Some details are simplified compared to the actual proof, so the sketch here should be viewed as illustrative only. We consider all of the solutions $n_1 < n_2 <\cdots < n_R$ to $n_i! = P(x_i)$ with $N \leq n_i < 2N$ and put the solutions in tuples $(n_i,n_{i+1},n_{i+2})$. We introduce a parameter $\mathcal{M}$ which is a small power of $N$. The number of tuples with $n_{i+2} - n_i$ larger than $\mathcal{M}$ are easily bounded, so we must bound the number of tuples with $n_{i+2} - n_i$ less than $\mathcal{M}$.

Following the strategy of \cite{BO1992}, the problem is transformed into one about simultaneous rational approximation of values of algebraic functions. From any tuple of solutions $(n_i,n_{i+1},n_{i+2})$ we construct a simultaneous rational approximation
\begin{align}\label{eq:intro very good simul approx}
\left|\prod_{j=1}^{J_k} \left(1 - \frac{j}{n_i} \right)^{-1/r} - \frac{p_k}{x_i} \right| \ll x_i^{-2+\varepsilon}, \ \ \ \ \ \ \ \ \ \ \ k = 1,2,
\end{align}
where $1\leq J_1 < J_2\leq \mathcal{M}$ are integers, and the $p_k$ are also integers. Write $\omega_1(1/n_i),\omega_2(1/n_i)$ for the algebraic numbers appearing in \eqref{eq:intro very good simul approx}. The denominator $x_i$ has size $\approx (N!)^{1/r}$, so the denominator is very large compared to the ``height'' of $\omega_k(1/n_i)$. Schmidt's subspace theorem (see \cite[Chapter 7]{BG2006}, for instance) says there are only finitely many denominators $q$ such that
\begin{align*}
\left|\omega_k(1/n_i) - \frac{r_k}{q} \right| \leq q^{-\frac{3}{2} - \epsilon}, \ \ \ \ \ \ \ k=1,2,
\end{align*}
and the large height of the rational numbers $p_k/x_i$ leads one to believe that no simultaneous rational approximation as in \eqref{eq:intro very good simul approx} is possible. The subspace theorem is ineffective, so the challenge is to work with effective Diophantine arguments in order to contradict \eqref{eq:intro very good simul approx}. The argument in \cite{BO1992} appeals to effective estimates of Osgood \cite{Osg1986} on lower bounds for linear forms in Siegel $G$-functions. Osgood's estimates, though effective, are not sufficiently explicit to prove a result like Theorem \ref{thm:main thm}, and it would be somewhat difficult to obtain the necessary modifications.

We choose not to frame our results in the language of $G$-functions and instead take a different, though related, approach. Namely, we proceed via the method of Pad\'e approximations (also sometimes called the \emph{hypergeometric method} or the \emph{Thue-Siegel method}, see e.g. \cite{Ben2001, Ben2008, Chu1983, CC1985}). Following a well-known method of Rickert \cite{Ric1993}, we need to construct rational numbers with certain special properties. These rational numbers will arise as the values of Pad\'e polynomials evaluated at a certain point. The Pad\'e polynomials $P_0,P_1,P_2$ of degree $\leq D$ are chosen to have small height (i.e. coefficients of small size), and so that the approximating form
\begin{align*}
R(x) = P_0(x) + P_1(x) \omega_1(x) + P_2(x) \omega_2(x)
\end{align*}
vanishes to high order at $x=0$. Here ``high order'' means $\geq (3-\epsilon_0) D$ for some parameter $\epsilon_0 \in (0,1)$ (the reader may pretend $\epsilon_0\approx \frac{1}{2}$ without losing much). We construct the polynomials via Siegel's lemma. The height of the polynomials increases as $\epsilon_0$ decreases, and we choose $\epsilon_0$ optimally to effect the right balance between high order of vanishing and low-height polynomials.

Actually, the initial Pad\'e polynomials $P_0,P_1,P_2$ are not sufficient for our purposes, since we cannot show they are suitably ``independent.'' Following \cite{Nag1995}, we construct from these initial Pad\'e polynomials two other sequences $P_i^{[k]}$ and $P_i^{\langle k \rangle}$ of Pad\'e polynomials. We can show the polynomials $P_i^{[k]}$ are suitably independent by using some classical arguments of Siegel \cite{Sieg1949}. The desired independence is that, for some $K$ and for $\alpha =1/n$ with $n$ some solution to $n! = P(x)$, the matrix
\begin{align*}
(P_i^{[k]}(\alpha))_{\substack{0 \leq k \leq K \\ 0\leq i \leq 2}}
\end{align*}
has full rank. Here it is important that $1,\omega_1,\omega_2$ are linearly independent over $\mathbb{C}(z)$, the field of rational functions with complex coefficients, and also that $\omega_1$ and $\omega_2$ satisfy relatively simple differential equations. The polynomials $P_i^{[k]}$ and $P_i^{\langle k \rangle}$ are related by a nonsingular transformation, so the polynomials $P_i^{\langle k \rangle}$ are also suitably independent. We switch to the polynomials $P_i^{\langle k \rangle}$ since it is easier to control the size of $P_i^{\langle k \rangle}(\alpha)$.

It is important that the matrix of rational numbers
\begin{align*}
(P_i^{\langle k \rangle}(\alpha))_{\substack{0 \leq k \leq K \\ 0\leq i \leq 2}}
\end{align*}
is not too large, i.e. that $K$ is not too large. The integer $K$ is essentially the order of vanishing of a certain determinant polynomial $\Delta(x)$ at $x=\alpha$. If there are many solutions $n$, then for some $n$ the polynomial $\Delta(x)$ vanishes to low order at $x =\alpha = 1/n$. On the other hand, if there are not many solutions then we are already done.

We have endeavored, for the convenience of the reader, to make the paper as self-contained as possible. Therefore, whenever we have cited a result from the literature we have also supplied a proof.

The outline of the rest of the paper is as follows. Section \ref{sec:notation} lays out the notation of the paper. In Section \ref{sec:reduction to key proposition} we reduce the proof of Theorem \ref{thm:main thm} to the proof of the technical Proposition \ref{prop:existence of good polynomials}. In Section \ref{sec:bin coeffs and initial Pade polys} we begin assembling the tools we need to study Pad\'e approximations, and we construct our initial Pad\'e polynomials. Section \ref{sec:indep Pade polys} contains the arguments showing that the polynomials $P_i^{[k]}$ are suitably independent. In Section \ref{sec:alternate pade polys} we switch to studying the polynomials $P_i^{\langle k \rangle}$, and all the pieces are assembled in Section \ref{sec:proof of key prop} for the proof of Proposition \ref{prop:existence of good polynomials}. In Section \ref{sec:future work} we offer some concluding thoughts about possible extensions and future work.

\section{Notation}\label{sec:notation}

Whenever $N$ appears in the paper, as in the statement of Theorem \ref{thm:main thm} or elsewhere, we always assume that $N$ is sufficiently large. We say $f\ll g$, $g \gg f$, or $f = O(g)$ for a nonnegative function $g$ if there is a positive constant $C$ such that $|f|\leq Cg$. If the constant $C$ depends on other parameters or quantities we generally indicate this with a subscript, e.g. $f\ll_P g$. We write $o(1)$ for a quantity which tends to zero as $N$ tends to infinity, and dependence of this quantity on other objects is sometimes indicated with a subscript.

We write $\mathbb{N},\mathbb{Z}, \mathbb{Q}$, and $\mathbb{C}$ for the set of positive integers, the set of integers, the set of rational numbers, and the set of complex numbers, respectively. Given a ring $R$, we write $R[x]$ for the ring of polynomials with coefficients in $R$ and $R(x)$ for the field of rational functions with coefficients in $R$.

Given rational functions $A,B$ and a differentiable function $f$, we inductively define the differential operators $(A \frac{d}{dx}+B)^j$ acting on $f$ by
\begin{align*}
\Big(A \frac{d}{dx}+B\Big)^0 f &= f(x), \\
\Big(A \frac{d}{dx}+B\Big)^{j+1} f &= A(x) \frac{d}{dx}\left(\Big(A \frac{d}{dx}+B\Big)^j f \right) + B(x) \left(\Big(A \frac{d}{dx}+B\Big)^j f \right).
\end{align*}

If $f(x)$ is holomorphic and not identically zero in a neighborhood of a point $\alpha \in \mathbb{C}$, we write $\text{ord}_{x=\alpha} f(x)$ for the largest nonnegative integer $n$ such that $\lim_{x\rightarrow \alpha} f(x)(x-\alpha)^{-n}$ is finite.

For any prime number $p$, we may write a nonzero rational number as $p^r \frac{a}{b}$, where $a,b,r \in \mathbb{Z}$ ($ab\neq 0$) and $p\nmid ab$. We define the $p$-adic valuation $v_p(p^r \frac{a}{b}) = r$.

Given a real number $y$, the floor $\lfloor y \rfloor$ of $y$ is the greatest integer which is $\leq y$. For a statement $S$, write $\mathbf{1}_S$ for the function which is 1 if $S$ is true and 0 if $S$ is false. We write $\text{deg}(P)$ for the degree of a nonzero polynomial $P$.

We collect here some unique notation and definitions that play a key role in the paper. They are introduced as needed in the paper, and we give here the location of their first appearance. The parameter $\epsilon_0$ is always a fixed real number in the interval $[\frac{1}{100}, \frac{99}{100}]$ (Proposition \ref{prop:existence of good polynomials}, Lemma \ref{lem:initial pade polynomials}). We have $\mathcal{M} = \lfloor N^\theta \rfloor$ for $\theta$ a small, positive number (see \eqref{eq:defn of mathcal M parameter}). The positive integers $\beta_1 < \beta_2\leq \mathcal{M}$ are multiples of the degree $r\geq 2$ (Lemma \ref{prop:not many close solutions}). The functions $\omega_0,\omega_1,\omega_2$ are defined in \eqref{eq:defn of omega functions}. The rational function $A_i(x), i \in \{0,1,2\}$, satisfies $\omega_i'(x) = A_i(x) \omega_i(x)$ (see \eqref{eq:defn of A funcs}). The polynomial $T(x) \in \mathbb{Z}[x]$ (see \eqref{eq:defn of T poly}) is such that $T(x)A_i(x) \in \mathbb{Z}[x]$ for $0\leq i \leq 2$.

\section{Reduction to Proposition \ref{prop:existence of good polynomials}}\label{sec:reduction to key proposition}

In this section we reduce the proof of Theorem \ref{thm:main thm} to the proof of Proposition \ref{prop:existence of good polynomials} below. Along the way we perform intermediate reductions and prove important supporting results. We deduce Theorem \ref{thm:main thm} from a slightly more technical result.

\begin{proposition}[Power saving for depressed polynomials]\label{prop:technical main thm}
Let $P\in \mathbb{Z}[x]$ be a polynomial of degree $r\geq 2$ with the coefficient of $x^{r-1}$ equal to zero. Let $s \in \mathbb{Z}\backslash \{0\}$ be fixed. Then for large $N$
\begin{align*}
\#\{N \leq n < 2N : s \cdot n! = P(x) \text{ for some } x \in \mathbb{N}\} \ll_{P,s} N^{33/34}.
\end{align*}
\end{proposition}

\begin{proof}[Proof of Theorem \ref{thm:main thm} assuming Proposition \ref{prop:technical main thm}]
The reduction is essentially \cite[Lemma 1]{BO1992}. Write
\begin{align*}
P(x) = \sum_{i=0}^r a_i x^i,
\end{align*}
where $a_i \in \mathbb{Z}$ and $a_r \neq 0$. We define $Q(x) = P(x - \frac{a_{r-1}}{ra_r}) \in \mathbb{Q}[x]$. Then $Q(x)$ has degree $r$, and the coefficient of $x^{r-1}$ in $Q$ is equal to zero. Moreover, the denominators of the coefficients of $Q$ are divisors of $(ra_r)^r$. Observe that
\begin{align*}
sn! = P(x) &= Q \left(x + \frac{a_{r-1}}{ra_r} \right) = \sum_{i=0}^r b_i \left(\frac{ra_r x + a_{r-1}}{ra_r} \right)^i
\end{align*}
for some rational numbers $b_i$ with denominators dividing $(ra_r)^r$ and $b_{r-1} = 0$. If we set $t = (ra_r)^{2r}$ then
\begin{align*}
stn! = t Q \left(x + \frac{a_{r-1}}{ra_r} \right) = R(ra_r x + a_{r-1}),
\end{align*}
where
\begin{align*}
R(y) &= \sum_{i=0}^r b_i (ra_r)^r \cdot (ra_r)^{r-i} y^i \in \mathbb{Z}[y],
\end{align*}
and the coefficient of $y^{r-1}$ in $R(y)$ is equal to zero. Note also that $R$ only depends on $P$. It follows that
\begin{align*}
&\#\{N \leq n < 2N : s \cdot n! = P(x) \text{ for some } x\in \mathbb{N}\} \\ 
&\leq \#\{N \leq n < 2N : st \cdot n! = R(y) \text{ for some } y \in \mathbb{N}\},
\end{align*}
and the desired bound then follows from Proposition \ref{prop:technical main thm}.
\end{proof}

We introduce the fundamental parameter
\begin{align}\label{eq:defn of mathcal M parameter}
\mathcal{M} := \lfloor N^\theta \rfloor, \ \ \ \ \ \ \ \ \ \ \frac{1}{1000}\leq  \theta \leq \frac{1}{20},
\end{align}
which serves to control the distance between solutions to $ s n! = P(x)$. Our method, very roughly speaking, is to show that there can only be few solutions $n_1,n_2$ with $sn_i! = P(x_i)$ and $|n_1-n_2| \leq \mathcal{M}$, and therefore most solutions are at distance $>\mathcal{M}$ from one another, hence there are $\lessapprox N\mathcal{M}^{-1}$ total solutions.

\begin{proposition}[Few solutions with small difference]\label{prop:not many close solutions}
Let $P\in \mathbb{Z}[x]$ be a polynomial of degree $r\geq 2$ with the coefficient of $x^{r-1}$ equal to zero. Let $s \in \mathbb{Z}\backslash \{0\}$ be fixed. Let $\epsilon>0$ be sufficiently small, and assume $N$ is sufficiently large in terms of $P,s,\epsilon$.  Let $\mathcal{M}$ be defined as in \eqref{eq:defn of mathcal M parameter}, and assume
\begin{align*}
\theta\leq 17 - 12\sqrt{2} - \epsilon.
\end{align*}
Let $\beta_1,\beta_2$ be positive integer multiples of $r$ with $r\leq \beta_1 < \beta_2 \leq \mathcal{M}$. Let $\mathcal{N} \in [N,2N)$ be an integer. Then
\begin{align*}
\sum_{\substack{\mathcal{N} \leq n < \mathcal{N} + \lfloor \frac{N}{\log N}\rfloor \\ N \leq n < 2N \\ sn! = P(x) \\ s(n-\beta_i)! = P(x_i), i=1,2}} 1 \leq \frac{N}{\mathcal{M}^3 (\log N)}.
\end{align*}
\end{proposition}

\begin{proof}[Proof of Proposition \ref{prop:technical main thm} assuming Proposition \ref{prop:not many close solutions}]
Let
\begin{align*}
N \leq n_1 < n_2 < \cdots < n_R < 2N
\end{align*}
be all the solutions to $s n! = P(x)$ with $n_i \in [N,2N)$, so that we wish to prove an upper bound on $R$. We may assume that $R > 10 r$, otherwise the desired upper bound trivially holds. By removing at most $2r$ solutions we may assume that $R$ is a multiple of $2r+1$, and then we put the solutions into $(2r+1)$-tuples
\begin{align*}
(n_i,n_{i+1},\ldots,n_{i+2r}).
\end{align*}
We separate tuples according to whether $n_{i+2r}-n_i$ is greater than $\mathcal{M}$ or less than $\mathcal{M}$. 

It is easy to show there are few solutions belonging to tuples with $n_{i+2r}-n_i > \mathcal{M}$. Indeed, the number of intervals $[n_i,n_{i+2r}]$ contained in $[N,2N)$ with $n_{i+2r}-n_i > \mathcal{M}$ is $\leq N/\mathcal{M}$ (by taking Lebesgue measures, for instance), and these intervals contain
\begin{align*}
&\leq (2r+1) \frac{N}{\mathcal{M}}
\end{align*}
solutions to $sn! = P(x)$.

We turn our attention to the tuples with $n_{i+2r}-n_i \leq \mathcal{M}$. By the pigeonhole principle, among the $2r+1$ elements in any given tuple there are at least three elements $n'' < n' < n$ which lie in the same residue class modulo $r$. We fix the differences $\beta_1 =n-n'$ and $\beta_2 = n-n''$, so that $2\leq r \leq \beta_1 < \beta_2 \leq \mathcal{M}$
with $\beta_1$ and $\beta_2$ multiples of $r$. There are trivially $\leq \mathcal{M}^2$ choices for the pair $(\beta_1,\beta_2)$, and therefore
\begin{align*}
R &\leq 2r + (2r+1) \frac{N}{\mathcal{M}} + (2r+1)\mathcal{M}^2 \max_{\substack{r\leq \beta_1 < \beta_2 \leq \mathcal{M} \\ r\mid \beta_i}} \sum_{\substack{N \leq n < 2N \\ sn! = P(x) \\ s(n-\beta_i)! = P(x_i), i=1,2}} 1.
\end{align*}
We cover $[N,2N)$ by short intervals of the form $[\mathcal{N},\mathcal{N} + \lfloor \frac{N}{\log N}\rfloor)$, which we can do with $\leq\log N+1$ short intervals. (It will be important later that
\begin{align}\label{eq:factorial nearly constant on short interval}
\mathcal{N}! \leq n! \leq (\mathcal{N}!)^{1+o(1)}
\end{align}
for any $n \in [\mathcal{N}, \mathcal{N} + \lfloor \frac{N}{\log N}\rfloor ]$, as can be checked with Stirling's formula.) Therefore
\begin{align*}
R &\ll_P \frac{N}{\mathcal{M}} + \mathcal{M}^2 (\log N) \max_{\substack{r\leq \beta_1 < \beta_2 \leq \mathcal{M} \\ r\mid \beta_i}} \, \max_{N \leq \mathcal{N} < 2N} \sum_{\substack{\mathcal{N} \leq n < \mathcal{N} + \lfloor \frac{N}{\log N}\rfloor \\ N \leq n < 2N \\ sn! = P(x) \\ s(n-\beta_i)! = P(x_i), i=1,2}} 1.
\end{align*}
By Proposition \ref{prop:not many close solutions} the sum over $n$ is $\leq \frac{N}{\mathcal{M}^3 \log N}$, and therefore $R\ll_P \frac{N}{\mathcal{M}}$. This finishes the proof upon taking, say, $\theta = 17-12\sqrt{2}-10^{-10}$.
\end{proof}

We have reduced matters to proving Proposition \ref{prop:not many close solutions}. As a first step, we show that any solution counted in the sum in Proposition \ref{prop:not many close solutions} gives rise to a strong simultaneous rational approximation.

\begin{lemma}[Solutions imply simultaneous rational approximation]\label{lem:solutions imply simul approx}
Let $P \in \mathbb{Z}[x]$ be a polynomial of degree $r\geq 2$ with the coefficient of $x^{r-1}$ equal to zero. Let $\mathcal{M}$ be defined as in \eqref{eq:defn of mathcal M parameter}, and let $\beta_1,\beta_2$ be positive integer multiples of $r$ with $r\leq \beta_1 < \beta_2 \leq \mathcal{M}$. Assume that
\begin{align*}
sn! = P(x), s(n-\beta_1)! = P(x_1), s(n-\beta_2)! = P(x_2)
\end{align*}
with $n \in [N,2N)$ and some positive integers $x,x_1,x_2$. Then for $i=1,2$ there exists $p_i \in \mathbb{Z}$ such that
\begin{align*}
\left| \prod_{j=1}^{\beta_i-1} \left(1 - \frac{j}{n} \right)^{-1/r} - \frac{p_i}{x} \right| \ll_P \frac{1}{x^{2-o(1)}}.
\end{align*}
\end{lemma}
\begin{proof}
The argument is that of \cite[p. 191]{BO1992}. For $i=1,2$, consider
\begin{align*}
\left|\frac{(n-\beta_i)!}{n!}x^r - x_i^r \right| &= \frac{1}{P(x)}|P(x_i)x^r - P(x)x_i^r|.
\end{align*}
Since the coefficient of $x^{r-1}$ in $P$ is zero, we see that
\begin{align*}
\frac{1}{P(x)}|P(x_i)x^r - P(x)x_i^r| \ll_P x^{r-2}.
\end{align*}
On the other hand, we may set $\zeta = e^{2\pi i /r}$ and then observe that
\begin{align*}
\left|\frac{(n-\beta_i)!}{n!}x^r - x_i^r \right| &= \prod_{j=0}^{r-1} \left|\zeta^j \left(\frac{(n-\beta_i)!}{n!} \right)^{1/r}x - x_i \right|.
\end{align*}
We claim that if $1\leq j \leq r-1$, then
\begin{align*}
\left|\zeta^j \left(\frac{(n-\beta_i)!}{n!} \right)^{1/r}x - x_i \right| > x^{1-o_P(1)}.
\end{align*}
If $r$ is odd, or if $r$ is even and $j \neq r/2$, then $|\text{Im}(\zeta^j)| \gg_r 1$. If $r$ is even and $j=r/2$ then $\zeta^j=-1$ and
\begin{align*}
\left|\zeta^j \left(\frac{(n-\beta_i)!}{n!} \right)^{1/r}x - x_i \right| = \left(\frac{(n-\beta_i)!}{n!} \right)^{1/r}x + x_i > \left(\frac{(n-\beta_i)!}{n!} \right)^{1/r}x.
\end{align*}
In any case, it suffices to prove that
\begin{align*}
\frac{n!}{(n-\beta_i)!} < x^{o_P(1)}.
\end{align*}
Crude estimates show $x \asymp_{P,s} (n!)^{1/r} \gg \left(\frac{N}{e} \right)^{N/r}$. On the other hand,
\begin{align*}
\frac{n!}{(n-\beta_i)!} = \prod_{j=0}^{\beta_i-1} (n-j)< (2N)^\mathcal{M},
\end{align*}
and the upper bound on $\mathcal{M}$ implies this is $\leq x^{o_P(1)}$.

By the claim, we have
\begin{align*}
\left|\left(\frac{(n-\beta_i)!}{n!} \right)^{1/r}x - x_i \right| \ll_P \frac{1}{x^{1-o(1)}},
\end{align*}
where the $o(1)$ quantity also depends on $P$. We have
\begin{align*}
\left(\frac{(n-\beta_i)!}{n!} \right)^{1/r} = n^{-\beta_i/r} \prod_{j=1}^{\beta_i-1} \left(1 - \frac{j}{n} \right)^{-1/r},
\end{align*}
where $n^{\beta_i/r}$ is an integer since $\beta_i$ is divisible by $r$. We easily check $n^{\beta_i/r} \leq x^{o(1)}$, so multiplying through by $n^{\beta_i/r}$ and dividing by $x$ gives
\begin{align*}
\left|\prod_{j=1}^{\beta_i-1} \left(1 - \frac{j}{n} \right)^{-1/r} - \frac{n^{\beta_i/r}x_i}{x} \right| &\ll \frac{1}{x^{2-o(1)}}. \qedhere
\end{align*}
\end{proof}

We make the definitions
\begin{align}\label{eq:defn of omega functions}
\begin{split}
\omega_0(x) &:= 1 \\
\omega_i(x) &:= \prod_{j=1}^{\beta_i-1} (1-jx)^{-1/r}, \  \ \ i=1,2, \ \ \ \ \ 2\leq r \leq \beta_1 < \beta_2 \leq \mathcal{M}, \ \ \ r \mid \beta_i.
\end{split}
\end{align}
Lemma \ref{lem:solutions imply simul approx} shows that solutions counted by the sum in Proposition \ref{prop:not many close solutions} give rise to strong simultaneous rational approximations to the algebraic values $\omega_i(1/n)$. We prove Proposition \ref{prop:not many close solutions} by contradiction. We assume that there are many solutions, and use this to find one particular value of $n$ with some desirable properties. For this particular value of $n$ we will be able to show that
\begin{align*}
\max\left\{\left|\omega_1(1/n) - \frac{p_1}{x} \right|, \left|\omega_2(1/n) - \frac{p_2}{x} \right| \right\} > \frac{1}{x^{2-\epsilon}},
\end{align*}
and this will contradict Lemma \ref{lem:solutions imply simul approx}. Proposition \ref{prop:existence of good polynomials} below shows we can do this provided we can show the existence of rational numbers with special properties. These rational numbers will arise from evaluating Pad\'e polynomials at $1/n$.

\begin{proposition}[Existence of approximating rationals]\label{prop:existence of good polynomials}
Let $P \in \mathbb{Z}[x]$ have degree $r\geq 2$ with the coefficient of $x^{r-1}$ equal to zero, and let $s \in \mathbb{Z}/\{0\}$ be fixed. Let $\frac{1}{100}\leq \epsilon_0 \leq \frac{99}{100}$ be a constant. Let $\mathcal{M}$ be given as in \eqref{eq:defn of mathcal M parameter}, and assume
\begin{align*}
\theta \leq \frac{\epsilon_0 (1-\epsilon_0)}{(3-\epsilon_0)(4-\epsilon_0)} - \epsilon
\end{align*}
with $\epsilon > 0$ sufficiently small. Assume $N$ is sufficiently large in terms of $P,s,\epsilon$. Let $\omega_i$ and $\beta_i$ be as in \eqref{eq:defn of omega functions}. Let $\mathcal{N} \in [N,2N)$ be an integer. Define
\begin{align*}
&c:= (2r\mathcal{N})^{\mathcal{M}^5}, \qquad\qquad\qquad\quad\ C := 2\mathcal{N}, \\
&u := (4r\mathcal{N})^{\mathcal{M}^4} 10^{\mathcal{M}\epsilon_0^{-1}}, \qquad\qquad U := \left( 4r^4 \mathcal{M}\right)^{(2-\epsilon_0)(3-\epsilon_0)\epsilon_0^{-1}}, \\
&w := (2r\mathcal{N})^{2\mathcal{M}^4} 4^{\mathcal{M}\epsilon_0^{-1}}, \qquad\qquad W := \frac{{\mathcal{N}}^{\, 3-\epsilon_0}}{(16r^4 \mathcal{M})^{2(3-\epsilon_0)\epsilon_0^{-1}}},
\end{align*}
and define
\begin{align*}
D &:= 1 + \left\lfloor \frac{\log \left(b cw \big((\mathcal{N} + \lfloor \frac{N}{\log N}\rfloor)!\big)^{1/r} \right)}{\log (W/C)} \right\rfloor,
\end{align*}
where $b = b(P,s)>0$ is a sufficiently large constant.

Assume
\begin{align*}
\sum_{\substack{\mathcal{N} \leq n < \mathcal{N} + \lfloor \frac{N}{\log N}\rfloor \\ N \leq n < 2N \\ sn! = P(x) \\ s(n-\beta_i)! = P(x_i), i=1,2}} 1 > \frac{N}{\mathcal{M}^3 (\log N)}.
\end{align*}
Then there exists $n_0 \in [\mathcal{N}, \mathcal{N} + \lfloor \frac{N}{\log N}\rfloor) \cap [N,2N)$ with $sn_0! = P(x), s(n_0-\beta_i)! = P(x_i)$, and rational numbers $p_{i,j}, 0\leq i,j\leq 2$, with the following properties:
\begin{enumerate}
\item The determinant of the matrix $(p_{i,j})_{0\leq i,j\leq 2}$ is nonzero.
\item There exists a positive integer $Z\leq cC^D$ such that $Zp_{i,j}\in\mathbb{Z}$ for $0\leq i,j\leq 2$.
\item $|p_{i,j}|\leq u U^D$.
\item For each $0\leq j \leq 2$ we have
\begin{align*}
\left|\sum_{i=0}^2 p_{i,j}\omega_i(1/n_0) \right|\leq wW^{-D}.
\end{align*}
\end{enumerate}
\end{proposition}

\begin{remark}
The numbers $c,u,w$ are negligible in comparison to the ``main term'' which arises, and so may be ignored on a first read. The numbers $C,U,W$ are all the size of fixed powers of $N$, and $D \asymp_{P,s} N$. The conditions on $\mathcal{M}$ and $\epsilon_0$ imply $W/C > 1$ (in fact, $\log(W/C) \gg \log N$).
\end{remark}

\begin{proof}[Proof of Proposition \ref{prop:not many close solutions} assuming Proposition \ref{prop:existence of good polynomials}]
The general structure of the argument is based on \cite[Lemma 2.1]{Ric1993}. Assume for contradiction that Proposition \ref{prop:not many close solutions} is false, so that
\begin{align*}
\sum_{\substack{\mathcal{N} \leq n < \mathcal{N} + \lfloor \frac{N}{\log N}\rfloor \\ N \leq n < 2N \\ sn! = P(x) \\ s(n-\beta_i)! = P(x_i), i=1,2}} 1 > \frac{N}{\mathcal{M}^3 (\log N)}.
\end{align*}
We may therefore apply Proposition \ref{prop:existence of good polynomials}, which gives the existence of a solution $n_0$ and rational numbers $p_{i,j}$ with all the properties stated there. Let $x$ be such that $sn_0! = P(x)$, and define
\begin{align*}
\rho := \max_{i=1,2}\left|\omega_i(1/n_0) - \frac{p_i}{x} \right|,
\end{align*}
where the integers $p_i$ are given by Lemma \ref{lem:solutions imply simul approx}. Further, Lemma \ref{lem:solutions imply simul approx} yields the bound
\begin{align*}
\rho \ll_{P} \frac{1}{x^{2-o(1)}}.
\end{align*}
We may write $P(x) = a_rx^r + O(x^{r-2})$, so $sn_0! = (1+o(1))a_r x^r$, and therefore by \eqref{eq:factorial nearly constant on short interval}
\begin{align*}
x \asymp_{P,s} (n_0!)^{1/r} = (\mathcal{N}!)^{1/r+o(1)},
\end{align*}
hence
\begin{align}\label{eq:rho upper bound}
\rho &\ll_{P,s} \big((\mathcal{N}!)^{1/r}\big)^{-2+o(1)}.
\end{align}

We now prove a lower bound for $\rho$ which contradicts the upper bound \eqref{eq:rho upper bound}. Write $p_0 = x$, and consider for $0\leq j\leq 2$ the sum
\begin{align*}
\sum_{i=0}^2 p_{i,j}p_i = -x\sum_{i=0}^2 p_{i,j}\left(\omega_i(1/n_0) - \frac{p_i}{x} \right) + x\sum_{i=0}^2 p_{i,j}\omega_i(1/n_0).
\end{align*}
Applying the triangle inequality and Proposition \ref{prop:existence of good polynomials} gives
\begin{align*}
\left|\sum_{i=0}^2 p_{i,j}p_i \right| &\leq 2x uU^D \rho + xwW^{-D}.
\end{align*}
Since the vector $(p_0,p_1,p_2)$ is nonzero and $\text{det}(p_{i,j})_{0\leq i,j\leq 2} \neq 0$ there exists some $j$ such that
\begin{align*}
\sum_{i=0}^2 p_{i,j}p_i \neq 0.
\end{align*}
It follows that there exists $0\leq j\leq 2$ with
\begin{align*}
\left|\sum_{i=0}^2 p_{i,j}p_i \right| \geq \frac{1}{Z}\geq \frac{1}{cC^D}.
\end{align*}
If
\begin{align}\label{eq:W/C ^D must be large enough}
\left( \frac{W}{C}\right)^D \geq 2cwx
\end{align}
then
\begin{align*}
\frac{1}{cC^D} - \frac{xw}{W^D} \geq \frac{1}{2cC^D},
\end{align*}
and \eqref{eq:W/C ^D must be large enough} holds since the definition of $D$ gives
\begin{align*}
\left( \frac{W}{C}\right)^D &\geq b cw \left(\left(\mathcal{N} + \left\lfloor \frac{N}{\log N}\right\rfloor\right)!\right)^{1/r}\geq bcw (n_0!)^{1/r} \geq 2cwx.
\end{align*}
Therefore $2x uU^D \rho \geq (2cC^D)^{-1}$, which implies
\begin{align*}
\rho &\geq \frac{1}{4cxu (CU)^D}\geq \frac{1}{4cCuU x} \left(bcw \left(\left(\mathcal{N} + \left\lfloor \frac{N}{\log N}\right\rfloor\right)!\right)^{1/r} \right)^{-\chi},
\end{align*}
where
\begin{align*}
\chi = \frac{\log(CU)}{\log (W/C)}.
\end{align*}
Since $c,u,w,C,U \leq (\mathcal{N}!)^{o(1)}$, we deduce that
\begin{align}\label{eq:rho lower bound}
\rho &\gg_{P,s} x^{-1}\left(\left(\mathcal{N} + \left\lfloor \frac{N}{\log N}\right\rfloor\right)!\right)^{-\chi/r-o(1)} \gg_{P,s} \big((\mathcal{N}!)^{1/r}\big)^{-1-\chi-o(1)},
\end{align}
where we have used \eqref{eq:factorial nearly constant on short interval} again. Then \eqref{eq:rho lower bound} contradicts \eqref{eq:rho upper bound} provided $\chi \leq 1 - \delta$ for some fixed $\delta > 0$ and $N$ is sufficiently large in terms of $P,s,\delta$. We have
\begin{align*}
\chi &= \frac{1 + \theta\frac{(2-\epsilon_0)(3-\epsilon_0)}{\epsilon_0}}{2-\epsilon_0 - \theta\frac{2(3-\epsilon_0)}{\epsilon_0}} + O_P\left(\frac{1}{\log N}\right),
\end{align*}
so we wish to impose the condition
\begin{align*}
\frac{1 + \theta\frac{(2-\epsilon_0)(3-\epsilon_0)}{\epsilon_0}}{2-\epsilon_0 - \theta\frac{2(3-\epsilon_0)}{\epsilon_0}} < 1.
\end{align*}
This is equivalent to
\begin{align*}
\theta < \frac{\epsilon_0 (1-\epsilon_0)}{(3-\epsilon_0)(4-\epsilon_0)},
\end{align*}
and for $\epsilon_0 \in (0,1)$ the right-hand side obtains its maximum value at
\begin{align*}
\epsilon_0 = 2 - \sqrt{2} = 0.5857\ldots,
\end{align*}
at which point we have
\begin{align*}
\frac{\epsilon_0 (1-\epsilon_0)}{(3-\epsilon_0)(4-\epsilon_0)} &= 17-12\sqrt{2} = 0.0294\ldots. \qedhere
\end{align*}
\end{proof}

\section{Binomial coefficients, and the initial Pad\'e polynomials}\label{sec:bin coeffs and initial Pade polys}

In this section we construct our initial Pad\'e polynomials. In order to construct these polynomials we first must understand the coefficients of the power series expansions
\begin{align}\label{eq:coeffs of power series exp}
\omega_i(x) &= \sum_{\ell \geq 0} b_{i,\ell} x^\ell, \ \ \ \ \ \ i=1,2.
\end{align}
It is easy to see the coefficients $b_{i,\ell}$ are rational numbers, but we need some knowledge about their denominators and sizes. Since $\omega_1$ and $\omega_2$ are products of binomial series, we begin by studying the denominators of binomial coefficients.

\begin{lemma}[Denominators of binomial coefficients]\label{lem:denom of binomial coeff}
Let $r\geq 2$ be an integer, and write
\begin{align*}
(1-y)^{-1/r} = \sum_{k \geq 0} (-1)^k {{-1/r} \choose k} y^k.
\end{align*}
Then the denominator of ${{-1/r} \choose k}$ divides $r^k \prod_{p\mid r} p^{\lfloor \frac{k}{p-1}\rfloor}$.
\end{lemma}
\begin{proof}
This is essentially contained in \cite[Lemma 4.1]{Chu1983}; we give a proof following \cite[Theorem 4.3]{Chu1983}. We may assume $k\geq 2$, since the conclusion of the lemma is straightforward for $k\leq 1$. By definition, we have
\begin{align*}
{{-1/r} \choose k} &= \frac{(-1/r)(-1/r-1)\cdots (-1/r-k+1)}{k!} = \frac{(-1)^k}{r^k} \frac{1}{k!}\prod_{j=1}^{k-1} (jr+1).
\end{align*}
We must determine the denominator of the rational number $\frac{1}{k!}\prod_{j=1}^{k-1} (jr+1)$, and therefore it suffices to study its $p$-adic valuation for each $p\leq k$. In particular, we obtain an upper bound on $v_p(k!)$ for $p\mid r$, and we show that 
\begin{align*}
v_p\left(\frac{1}{k!}\prod_{j=1}^{k-1} (jr+1)\right) \geq 0
\end{align*}
for $p\nmid r$. 

First, consider primes $p\mid r$. Then $p \nmid jr+1$, and
\begin{align*}
v_p(k!) =\sum_{j\geq 1} \left\lfloor \frac{k}{p^j}\right\rfloor \leq \sum_{j\geq 1} \frac{k}{p^j} =\frac{k}{p-1},
\end{align*}
so $v_p(k!) \leq \lfloor \frac{k}{p-1}\rfloor$.

Now consider primes $p \nmid r$ with $p\leq k$. Let $\mu$ be the integer such that $kr < p^\mu \leq pkr$, and choose $1\leq \ell < p^\mu$ such that $r\ell \equiv -1 \pmod{p^\mu}$. The congruence implies $p\nmid \ell$. Observe that since $p^\mu \mid (1+r\ell)$ we have
\begin{align*}
r\ell \geq p^\mu -1 \geq kr,
\end{align*}
so $\ell \geq k$. We also have $jr+1 \leq (k-1)r+1 < kr < p^\mu$, so $v_p(jr+1) < \mu$ for $1\leq j \leq k-1$. We deduce that $v_p(jr+1) = v_p(jr-r\ell) = v_p(j-\ell)$, and therefore
\begin{align*}
v_p \left(\frac{1}{k!}\prod_{j=1}^{k-1} (jr+1) \right) = v_p \left(\frac{1}{k!}\prod_{j=1}^{k-1} (\ell-j) \right).
\end{align*}
If $\ell=k$ then this is equal to $v_p((k-1)!/k!) = v_p(1/k)=0$. If $\ell > k$ then
\begin{align*}
v_p \left(\frac{1}{k!}\prod_{j=1}^{k-1} (\ell-j) \right) &=v_p \left(\frac{1}{\ell} {{\ell}\choose k} \right) \geq 0. \qedhere
\end{align*}
\end{proof}

With Lemma \ref{lem:denom of binomial coeff} in hand, we can state the result we need regarding the coefficients of the Taylor series of $\omega_i(x)$.
\begin{lemma}[Coefficients of $\omega$ functions]\label{lem:coeffs of omega i}
Let $\beta,r \geq 2$ be integers, and write
\begin{align*}
\omega(x)=  \prod_{j=1}^{\beta-1} (1-jx)^{-1/r} = \sum_{\ell \geq 0} b_{\ell}x^\ell.
\end{align*}
Then $b_{\ell}$ is a rational number with denominator dividing $r^{2\ell}$, and $|b_{\ell}| \leq 2^{\beta}(2\beta)^\ell$.
\end{lemma}
\begin{proof}
We multiply the binomial series together and collect powers of $x$ to obtain
\begin{align*}
\omega(x) &= \sum_{\ell \geq 0} x^\ell \sum_{\substack{k_1+\cdots + k_{\beta-1}=\ell \\ k_j \geq 0}} \prod_{j=1}^{\beta-1} (-1)^{k_j} {{-1/r}\choose k_j} j^{k_j},
\end{align*}
so
\begin{align*}
b_{\ell} = \sum_{\substack{k_1+\cdots + k_{\beta-1}=\ell \\ k_j \geq 0}} \prod_{j=1}^{\beta-1} (-1)^{k_j} {{-1/r}\choose k_j} j^{k_j}.
\end{align*}
By Lemma \ref{lem:denom of binomial coeff} the denominator of ${{-1/r}\choose k_j}$ divides $r^{k_j}\prod_{p\mid r}p^{\lfloor \frac{k_j}{p-1}\rfloor} \mid r^{2k_j}$, and since $k_1+\cdots+k_{\beta-1} = \ell$ it follows that the denominator of $b_{\ell}$ divides $r^{2\ell}$. This verifies the first claim of the lemma.

For the second claim, an easy induction with the relation ${{-1/r}\choose {k+1}} = -{{-1/r}\choose k}\frac{k+\frac{1}{r}}{k+1}$
shows that $\left|{{-1/r}\choose k} \right|\leq 1$ for every $k\geq 0$, and therefore
\begin{align*}
|b_{\ell}| &\leq \sum_{\substack{k_1+\cdots + k_{\beta-1}=\ell \\ k_j \geq 0}} \prod_{j=1}^{\beta-1} j^{k_j}.
\end{align*}
The trivial bound $j\leq \beta$ yields
\begin{align*}
|b_{\ell}| &\leq \beta^\ell \sum_{\substack{k_1+\cdots + k_{\beta-1}=\ell \\ k_j \geq 0}}1 = \beta^\ell {{\ell +\beta-2} \choose {\beta-2}} \leq \beta^\ell 2^{\ell+\beta}. \qedhere
\end{align*}
\end{proof}

We use Siegel's lemma in order to construct our initial Pad\'e polynomials. We need only a basic version of Siegel's lemma, and do not require the sharpest possible estimates. In particular, using a slightly sharper form of Siegel's lemma would not improve the final quality of the results (but see a further comment in Section \ref{sec:future work}).

\begin{lemma}[Siegel's lemma]\label{lem:siegels lemma}
Let
\begin{align*}
a_{1,1}X_1 + \cdots + a_{1,N}X_N &= 0 \\
\vdots \ \ \ \ \ \ \ \ \ \ \  &\ \\
a_{M,1}X_1 + \cdots + a_{M,N}X_N &= 0
\end{align*}
be a system of $M\geq 1$ homogeneous linear equations in $N>M$ variables. Assume $a_{i,j} \in \mathbb{Z}$ with $|a_{i,j}| \leq A$, where $A \geq 1$. Then there exists a solution to the system of equations in integers $X_i$ with not all $X_i$ equal to zero and
\begin{align*}
|X_i| \leq (3AN)^{\frac{M}{N-M}}.
\end{align*}
\end{lemma}
\begin{proof}
Let $\Phi = (a_{i,j})$ be the matrix of the coefficients of the system of equations, which we view as a linear map $\Phi:\mathbb{Z}^N \rightarrow \mathbb{Z}^M$. Let $W$ be a positive integer, and define $\mathcal{B} = [1,W]^N \subset \mathbb{Z}^N$. The cardinality of $\mathcal{B}$ is $W^N$, and the image of $\mathcal{B}$ under $\Phi$ is contained in $[-ANW,ANW]^M \cap \mathbb{Z}^M$, which has cardinality $\leq (2ANW +1)^M\leq (3ANW)^M$. If $W^N > (3ANW)^M$, then the pigeonhole principle implies there are distinct $b,b' \in \mathcal{B}$ such that $\Phi b = \Phi b'$, in which case the nonzero vector $b-b' = (X_1,\ldots,X_N)$ is a solution to the system of homogeneous equations with $|X_i| \leq W-1$. We have $W^N > (3ANW)^M$ if $W > (3AN)^{\frac{M}{N-M}}$, and we may choose an integral $W$ satisfying this inequality with $W \leq 1 + (3AN)^{\frac{M}{N-M}}$.
\end{proof}

We are ready to construct our initial Pad\'e polynomials.

\begin{lemma}[Initial Pad\'e polynomials]\label{lem:initial pade polynomials}
Let $D \geq 10$ be a positive integer, and let $\frac{1}{100} \leq \epsilon_0 \leq \frac{99}{100}$. Define $\mathcal{O} := \lfloor (3-\epsilon_0)D\rfloor$. Let $\omega_i$ and $\beta_i$ be as in \eqref{eq:defn of omega functions}. There exist polynomials $P_0,P_1,P_2\in \mathbb{Z}[x]$ of degree $\leq D$ with the following properties:
\begin{enumerate}
\item At least one of the polynomials $P_i$ is not identically zero.
\item The coefficient of $x^v$ in $P_i$ has absolute value $\leq 4^{\mathcal{M}\epsilon_0^{-1}}(4r^4\mathcal{M})^{(2-\epsilon_0)(3-\epsilon_0)\epsilon_0^{-1} D}(2\mathcal{M})^v$.
\item The function
\begin{align*}
R(x) := \sum_{i=0}^2 P_i(x) \omega_i(x) = \sum_{v=0}^\infty r_v x^v
\end{align*}
satisfies $r_v = 0$ for $0\leq v \leq \mathcal{O}$, and the coefficients $r_v$ satisfy
\begin{align*}
|r_v| &\leq 4^{\mathcal{M}\epsilon_0^{-1}}(4r^4\mathcal{M})^{(2-\epsilon_0)(3-\epsilon_0)\epsilon_0^{-1} D}(2\mathcal{M})^v.
\end{align*}
\end{enumerate}
\end{lemma}

\begin{remark}
The only consequential parts of the bounds in Lemma \ref{lem:initial pade polynomials} are the parts involving powers of $\mathcal{M}$. That is, one should interpret the bounds as
\begin{align*}
|p_{i,k}|, |r_v| &\lessapprox \mathcal{M}^{(2-\epsilon_0)(3-\epsilon_0)\epsilon_0^{-1} D}\mathcal{M}^v.
\end{align*}
\end{remark}

\begin{proof}
The condition that $r_v=0$ for $0\leq v\leq \mathcal{O}$ can be encoded as a system of homogeneous linear equations. We write $P_i(x) = \sum_{k=0}^D p_{i,k}x^k$, where the integral coefficients $p_{i,k}$ are to be determined (the integers $p_{i,k}$ should not be confused with the rational numbers $p_{i,j}$ in Proposition \ref{prop:existence of good polynomials}). By comparing coefficients we derive
\begin{align*}
r_v &= p_{0,v}\cdot\mathbf{1}_{v\leq D}  + \sum_{i=1}^2 \sum_{0\leq k\leq \min(v,D)}p_{i,k}b_{i,v-k},
\end{align*}
where the rational numbers $b_{i,\ell}$ are defined via \eqref{eq:coeffs of power series exp}. In order to apply Siegel's lemma we need the coefficients of the linear system (in this case, the numbers $b_{i,v-k}$) to be integers. By Lemma \ref{lem:coeffs of omega i} the denominator of $b_{i,\ell}$ divides $r^{2\ell}$, so if we multiply the equations through by $r^{2\mathcal{O}}$ we have
\begin{align*}
r^{2\mathcal{O}} r_v &= r^{2\mathcal{O}} p_{0,v}\cdot\mathbf{1}_{v\leq D}  + \sum_{i=1}^2 \sum_{0\leq k\leq \min(v,D)}p_{i,k}\cdot r^{2\mathcal{O}} b_{i,v-k},
\end{align*}
where $r^{2\mathcal{O}} b_{i,v-k} \in \mathbb{Z}$ for $0\leq v\leq \mathcal{O}$. We therefore wish to find a nontrivial solution in integers $p_{i,k}$ to the system of homogeneous linear equations
\begin{align*}
r^{2\mathcal{O}} p_{0,v}\cdot\mathbf{1}_{v\leq D}  + \sum_{i=1}^2 \sum_{0\leq k\leq \min(v,D)}p_{i,k}\cdot r^{2\mathcal{O}} b_{i,v-k} = 0, \ \ \ \ \ 0\leq v \leq \mathcal{O}.
\end{align*}
This is in the form to directly apply Lemma \ref{lem:siegels lemma}, but we obtain a better bound for the coefficients $p_{i,k}$ if we utilize the special form of the linear system, namely, the simple manner in which the variables $p_{0,k}$ appear. Specifically, we use Siegel's lemma only to solve the equations for $D+1 \leq v \leq \mathcal{O}$, and then we choose the coefficients $p_{0,k}$ to solve the equations with $0\leq v\leq D$. 

The system of equations with $D+1\leq v \leq \mathcal{O}$ is a system of $\mathcal{O}-D$ equations in the $2D+2$ variables $p_{1,k},p_{2,k}$. By Lemmas \ref{lem:coeffs of omega i} and \ref{lem:siegels lemma} there is a nontrivial solution to this system of equations with
\begin{align*}
|p_{i,k}| &\leq \left(3 (2D+2)r^{2\mathcal{O}} 2^\mathcal{M} (2\mathcal{M})^{\mathcal{O}} \right)^{\frac{\mathcal{O}-D}{2D+2-(\mathcal{O}-D)}} \leq \left(6(D+1) 2^\mathcal{M} (2r^2\mathcal{M})^{\mathcal{O}} \right)^{(2-\epsilon_0)\epsilon_0^{-1}}.
\end{align*}
Since $D \geq 10$, elementary calculus shows
\begin{align*}
6(D+1) = \Big(\big(6(D+1)\big)^{1/\mathcal{O}}\Big)^{\mathcal{O}} \leq \Big(\big(6(D+1)\big)^{\frac{1}{2D}}\Big)^{\mathcal{O}}\leq \left( \frac{3}{2}\right)^\mathcal{O},
\end{align*}
and therefore $|p_{i,k}| \leq 2^{(2-\epsilon_0)\epsilon_0^{-1}\mathcal{M}} (3r^2\mathcal{M})^{(2-\epsilon_0)\epsilon_0^{-1} \mathcal{O}}$ for $1\leq i \leq 2$.

We next modify the coefficients $p_{1,k},p_{2,k}$ in order to ensure we can choose $p_{0,k}$ to solve the equations for $0\leq v\leq D$. Since $\{p_{i,k}\}$ solves a system of homogeneous linear equations, so does the dilation $\{\lambda p_{i,k}\}$ for any nonzero integer $\lambda$. We take $\lambda = r^{2\mathcal{O}}$ and then change variables $r^{2\mathcal{O}} p_{i,k} \rightarrow p_{i,k}$ so that each integer $p_{1,k},p_{2,k}$, is divisible by $r^{2\mathcal{O}}$, and
\begin{align}\label{eq:bound for coeffs pik}
|p_{i,k}| &\leq 2^{(2-\epsilon_0)\epsilon_0^{-1}\mathcal{M}} (3r^4\mathcal{M})^{(2-\epsilon_0)\epsilon_0^{-1} \mathcal{O}}, \ \ \ \ \ \ i=1,2.
\end{align}
Furthermore, at least one of the coefficients $p_{i,k}$ is nonzero, and therefore at least one of $P_1$ or $P_2$ is not identically zero.

We now show there is a suitable choice of $p_{0,k}$ so that
\begin{align*}
r^{2\mathcal{O}} p_{0,v}  + \sum_{i=1}^2 \sum_{0\leq k\leq v}p_{i,k}\cdot r^{2\mathcal{O}} b_{i,v-k} = 0, \ \ \ \ \ \ \ \ 0 \leq v \leq D.
\end{align*}
Since $r^{2\mathcal{O}}$ divides each $p_{i,k}$ with $1\leq i\leq 2$, it suffices to choose
\begin{align*}
p_{0,v} &= -\sum_{i=1}^2 \sum_{0\leq k\leq v}\frac{p_{i,k}}{r^{2\mathcal{O}}}\cdot r^{2\mathcal{O}} b_{i,v-k} \in \mathbb{Z},
\end{align*}
so that by Lemma \ref{lem:coeffs of omega i} and \eqref{eq:bound for coeffs pik} we have
\begin{align*}
|p_{0,v}| &\leq 2(D+1) \cdot 2^{(2-\epsilon_0)\epsilon_0^{-1}\mathcal{M}} (3r^4\mathcal{M})^{(2-\epsilon_0)\epsilon_0^{-1} \mathcal{O}} \cdot 2^{\mathcal{M}} (2\mathcal{M})^v \\
&\leq 4^{\mathcal{M}\epsilon_0^{-1}}(4r^4\mathcal{M})^{(2-\epsilon_0)(3-\epsilon_0)\epsilon_0^{-1} D}(2\mathcal{M})^v.
\end{align*}

It remains to bound the coefficients $r_v$. We have $r_v=0$ for $v\leq \mathcal{O}$, so we may assume $v > \mathcal{O}$. Then
\begin{align*}
r_v &= \sum_{i=1}^2 \sum_{0\leq k\leq D}p_{i,k}b_{i,v-k},
\end{align*}
so by \eqref{eq:bound for coeffs pik} and Lemma \ref{lem:coeffs of omega i} again
\begin{align*}
|r_v| &\leq 4^{\mathcal{M}\epsilon_0^{-1}}(4r^4\mathcal{M})^{(2-\epsilon_0)(3-\epsilon_0)\epsilon_0^{-1} D}(2\mathcal{M})^v. \qedhere
\end{align*}
\end{proof}

\section{Independent Pad\'e polynomials}\label{sec:indep Pade polys}

We constructed Pad\'e polynomials $P_i$ in Lemma \ref{lem:initial pade polynomials} by appealing to Siegel's lemma. However, these polynomials are not immediately suitable for our purposes because we cannot guarantee that the polynomials $P_i$ are independent. In this section we introduce the polynomials $P_i^{[k]}$ and show they possess the necessary independence (see Lemma \ref{lem:full rank for square bracket polys}). The fact that the functions $\omega_i$ have simple differential equations plays a vital role.

\begin{lemma}[Differential equation for $\omega$]\label{lem:diff eq for omega funcs}
Let $\beta,r \geq 2$ be integers, and let
\begin{align*}
\omega(x) = \prod_{j=1}^{\beta-1} (1-jx)^{-1/r}.
\end{align*}
Then $\omega'(x) = A(x) \omega(x)$, where
\begin{align*}
A(x) &= \sum_{j=1}^{\beta-1} \frac{j}{r}\frac{1}{1-jx} \in \mathbb{Q}(x).
\end{align*}
\end{lemma}

Lemma \ref{lem:diff eq for omega funcs} implies $\omega_i'(x) = A_i(x) \omega_i(x)$, where
\begin{align}\label{eq:defn of A funcs}
\begin{split}
A_0(x) &:= 0, \\
A_i(x) &:= \sum_{j=1}^{\beta_i-1} \frac{j}{r}\frac{1}{1-jx}, \ \ \ \ \ \ \ \ i=1,2.
\end{split}
\end{align}
We also define the polynomial $T(x) \in \mathbb{Z}[x]$ by
\begin{align}\label{eq:defn of T poly}
T(x) := r\prod_{j=1}^{\mathcal{M}}(1-jx),
\end{align}
so that $T(x)A_i(x) \in \mathbb{Z}[x]$ for $0\leq i \leq 2$. Note that $\text{deg}(T) = \mathcal{M}$ and $\text{deg}(TA_i)\leq \mathcal{M}$.

Given polynomials $P_i$ and the function $R$ as in Lemma \ref{lem:initial pade polynomials}, we define for $k\geq 0$ the functions
\begin{align}\label{eq:defn of R square bracket}
R^{[k]} := \left(T \frac{d}{dx} \right)^k R
\end{align}
and polynomials
\begin{align}\label{eq:defn of P square bracket}
P_i^{[k]} := \left(T \left(\frac{d}{dx} + A_i \right) \right)^k P_i.
\end{align}
We then have the following lemma, which one can easily prove by induction (see also \cite[Lemma 1.4]{Nag1995}).

\begin{lemma}[Properties of $P_i^{[k]}$ and $R^{[k]}$]\label{lem:props of square bracket Pade polynomials}
Let $P_i^{[k]}$ and $R^{[k]}$ be defined as in \eqref{eq:defn of R square bracket} and \eqref{eq:defn of P square bracket}. The following are true for $k\geq 0$\textup{:}
\begin{enumerate}
\item $P_i^{[k]}(x) \in \mathbb{Z}[x]$,
\item $\textup{deg}(P_i^{[k]}) \leq D+k\mathcal{M}$,
\item $R^{[k]}(x) = \sum_{i=0}^2 P_i^{[k]}(x) \omega_i(x)$.
\end{enumerate}
\end{lemma}

In preparation for the study of a certain determinant (see \eqref{eq:defn of det poly} below), we need a linear independence result.

\begin{lemma}[Linear independence of $\omega$ functions]\label{lem:omega funcs are lin indep}
Suppose $U,V,W \in \mathbb{C}[x]$ are such that
\begin{align*}
U(x)\omega_0(x) + V(x)\omega_1(x) + W(x)\omega_2(x) = 0.
\end{align*}
Then $U=V=W=0$.
\end{lemma}
\begin{proof}
Assume for contradiction that $U\omega_0 + V\omega_1 + W\omega_2 = 0$ with not all of $U,V,W$ equal to zero. Suppose first that $W = 0$ but $UV \neq 0$. By renaming we find that $\omega_1 = \frac{U}{V} \in \mathbb{C}(x)$. Taking $r$th powers and rearranging gives
\begin{align*}
U(x)^r \prod_{j=1}^{\beta_1-1}(1-jx) = V(x)^r,
\end{align*}
but this is a contradiction, since $1-x$ divides the left-hand side with order $\equiv 1 \pmod{r}$, but divides the right-hand side with order $\equiv 0 \pmod{r}$.

We may therefore assume that $W\neq 0$. Then
\begin{align}\label{eq:false linear relation}
\omega_2 = \gamma + \delta \omega_1,
\end{align}
with $\gamma, \delta \in \mathbb{C}(x)$ not both zero. By the argument in the case $W=0$ above we must have $\delta \neq 0$. Multiplying \eqref{eq:false linear relation} by $A_2$ (recall \eqref{eq:defn of A funcs}) gives $A_2 \omega_2 = \gamma A_2 + \delta A_2 \omega_1$. On the other hand, differentiating \eqref{eq:false linear relation} via Lemma \ref{lem:diff eq for omega funcs} gives
\begin{align*}
A_2 \omega_2 = \gamma' + (\delta' + \delta A_1)\omega_1,
\end{align*}
and equating the two expressions for $A_2 \omega_2$ yields $\gamma A_2 + \delta A_2 \omega_1 = \gamma' + (\delta' + \delta A_1)\omega_1$, or
\begin{align}\label{eq:implies omega 1 is rational}
\gamma A_2 - \gamma' &= (\delta' + \delta (A_1 - A_2)) \omega_1.
\end{align}
It is not possible that $\delta' + \delta (A_1 - A_2) = 0$, since $\delta'/\delta$ is an integer linear combination of rational functions $(x-\zeta)^{-1}, \zeta \in \mathbb{C}$, but $A_1 - A_2$ is not. Hence $\delta' + \delta (A_1 - A_2) \neq 0$, so dividing through in \eqref{eq:implies omega 1 is rational} implies $\omega_1 \in \mathbb{C}(x)$, but we have already seen this is impossible.
\end{proof}

We introduce the determinant polynomial $\Delta(x) \in \mathbb{Z}[x]$, which is defined as
\begin{align}\label{eq:defn of det poly}
\Delta := \text{det}
\begin{pmatrix}
P_0^{[0]} & P_1^{[0]} & P_2^{[0]} \\
P_0^{[1]} & P_1^{[1]} & P_2^{[1]} \\
P_0^{[2]} & P_1^{[2]} & P_2^{[2]}
\end{pmatrix}.
\end{align}
Observe that the first row of the matrix here is given by the initial Pad\'e polynomials constructed in Lemma \ref{lem:initial pade polynomials}. We have the following important result, due to Siegel.

\begin{lemma}[Determinant polynomial is nonzero]\label{lem:det poly not identically zero}
Let $P_i^{[k]}$ be defined as in \eqref{eq:defn of P square bracket} and let $\Delta$ be defined as in \eqref{eq:defn of det poly}. Assume that $P_i(x)\neq 0$ for each $0\leq i \leq 2$. Then $\Delta(x)$ is not identically zero.
\end{lemma}
\begin{proof}
This is \cite[Chapter 2, Lemma 4]{Sieg1949}. Suppose for contradiction that each polynomial $P_i(x)$ is nonzero, but that $\Delta(x) = 0$. We may consider the entries of the matrix in the definition of $\Delta$ as being in the field $\mathbb{Q}(x)$, and then $\Delta = 0$ implies there is a nontrivial linear relation between the rows. By multiplying through to clear denominators, it follows that there exists $\mu \in \{1,2\}$ and polynomials $B_0,\ldots,B_\mu \in \mathbb{Z}[x]$ with $B_\mu \neq 0$ such that
\begin{align*}
\sum_{k=0}^\mu B_k(x) P_{i}^{[k]}(x) = 0
\end{align*}
for each $0\leq i \leq 2$. Lemma \ref{lem:props of square bracket Pade polynomials} implies $\sum_{k=0}^\mu B_k(x) R^{[k]}(x) = 0$,
since
\begin{align*}
\sum_{k=0}^\mu B_k(x) R^{[k]}(x) = \sum_{k=0}^\mu B_k(x) \sum_{i=0}^2 P_i^{[k]}(x) \omega_i(x) = \sum_{i=0}^2 \omega_i(x) \sum_{k=0}^\mu B_k(x) P_{i}^{[k]}(x) = 0.
\end{align*}
The condition $\sum_{k=0}^\mu B_k(x) R^{[k]}(x) = 0$
is equivalent to
\begin{align*}
\sum_{j=0}^\mu C_j(x) R^{(\mu-j)}(x) = 0,
\end{align*}
where $C_0(x) = B_\mu(x)T(x)^{\mu} \neq 0$ and the other $C_j$ are polynomials. Each of the three functions $P_i(x) \omega_i(x)$ is nonzero, and since $\left( T \frac{d}{dx}\right)^k (P_i \omega_i) = P_i^{[k]}\omega_i$, each function is a solution to the homogeneous linear differential equation
\begin{align*}
\sum_{k=0}^\mu B_k(x) \left( T \frac{d}{dx}\right)^k =  \sum_{j=0}^\mu C_j(x) \left(\frac{d}{dx} \right)^{\mu-j}=0
\end{align*}
of order $\mu < 3$. Therefore, there must exist a linear relation
\begin{align*}
\sum_{i=0}^2 c_i P_i(x) \omega_i(x) = 0,
\end{align*}
where the $c_i$ are complex constants, not all of which are zero (since the solutions to a homogeneous linear differential equation of order $\mu$ form a vector space of dimension $\mu$). By Lemma \ref{lem:omega funcs are lin indep} this implies $c_iP_i = 0$ for each $i$, and this is a contradiction.
\end{proof}

\begin{remark}
We can use Siegel's simple and effective argument to show $\Delta \neq 0$ since the functions $\omega_i$ have simple differential equations. If the differential equations were more complicated we might need more complicated arguments, such as those of Shidlovskii \cite[Chapter 3]{Shi1989}, which are more difficult to make effective.
\end{remark}

\begin{lemma}[$P_i^{[k]}(\alpha)$ matrix has full rank]\label{lem:full rank for square bracket polys}
Let $P_i^{[k]}$ be defined as in \eqref{eq:defn of P square bracket} and let $\Delta$ be defined as in \eqref{eq:defn of det poly}. Let $\alpha \in \mathbb{C}\backslash\{0\}$ be such that $T(\alpha)\neq 0$. Assume $D > \frac{\mathcal{M} + 2}{1-\epsilon_0}$.  Then $\Delta \neq 0$. 

If $\Delta \neq 0$, write $a=\textup{ord}_{x=\alpha}\Delta$ ($a$ is a nonnegative integer). Then the matrix
\begin{align*}
(P_i^{[k]}(\alpha))_{\substack{0\leq k \leq a+2 \\ 0\leq i \leq 2}}
\end{align*}
has rank three.
\end{lemma}
\begin{proof}
The result is essentially that of \cite[Chapter 2, Lemma 5]{Sieg1949}. By Lemma \ref{lem:initial pade polynomials}, at least one of the polynomials $P_i(x) = P_i^{[0]}(x)$, $0\leq i \leq 2$, is nonzero. It cannot be that exactly one of the polynomials $P_i$ is nonzero, since in that case $R(x) = P(x) \omega(x)$, but $R(x)$ is a nonzero function vanishing at zero to order $\geq \mathcal{O} \geq (3-\epsilon_0)D - 1$ (Lemma \ref{lem:initial pade polynomials}), and $P(x)\omega(x)$ vanishes at zero to order at most $D$ (observe that $\omega_i(0) \neq 0$ for $0\leq i \leq 2$). Hence we obtain a contradiction since $D\geq 1$ and $\epsilon_0 < 1$.

Now suppose for contradiction that exactly two of the polynomials $P_i$ are nonzero. By temporarily relabeling we may assume that $P_0P_1\neq 0$ and $P_2 = 0$. Following the proof of Lemma \ref{lem:det poly not identically zero}, we find that
\begin{align*}
\widetilde{\Delta} = \text{det}
\begin{pmatrix}
P_0 & P_1 \\
P_0^{[1]} & P_1^{[1]}
\end{pmatrix}
\neq 0,
\end{align*}
so that $\widetilde{\Delta}$ is a nonzero polynomial with degree $\leq 2D + \mathcal{M}$. We easily check the identity
\begin{align}\label{eq:tilde delta identity}
\widetilde{\Delta} \omega_0 &= P_1^{[1]} (P_0\omega_0 + P_1 \omega_1) - P_1 (P_0^{[1]}\omega_0 + P_1^{[1]}\omega_1) = P_1^{[1]} R^{[0]} - P_1 R^{[1]},
\end{align}
and observe that the left-hand side vanishes at $x=0$ to order $\leq 2D + \mathcal{M}$, while by Lemmas \ref{lem:initial pade polynomials} and \ref{lem:props of square bracket Pade polynomials} the right-hand side vanishes at $x=0$ to order $\geq \mathcal{O} - 1 \geq (3-\epsilon_0)D - 2$. This is a contradiction if $D > \frac{\mathcal{M} + 2}{1-\epsilon_0}$. It follows that each polynomial $P_0,P_1,P_2$ is nonzero, and by Lemma \ref{lem:det poly not identically zero} we have $\Delta \neq 0$.

 In the spirit of \eqref{eq:tilde delta identity}, we have the identity
\begin{align*}
\Delta \omega_k &= \sum_{\ell=0}^2\Delta_{k,\ell}R^{[\ell]},
\end{align*}
where $\Delta_{k,\ell}$ is the $(k,\ell)$-cofactor of the matrix in the definition \eqref{eq:defn of det poly} of $\Delta$. Now apply $(T\frac{d}{dx})^J$ to both sides of this identity. It follows by induction that
\begin{align*}
T(x)^J \Delta^{(J)}(x) \omega_k(x) + \sum_{j=0}^{J-1}\Delta^{(j)}(x) L_{J,k,j}(x) = \sum_{\ell=0}^{J+2} M_{J,k,\ell}(x) R^{[\ell]}(x),
\end{align*}
where the $L_{J,k,j}$ are linear forms in $\omega_0,\omega_1,\omega_2$ with polynomial coefficients, and the $M_{J,k,\ell}$ are polynomials. By comparing the coefficients for $\omega_i$ on both sides and applying Lemma \ref{lem:omega funcs are lin indep}, we deduce the stronger identity
\begin{align*}
T(x)^J \Delta^{(J)}(x) y_k + \sum_{j=0}^{J-1}\Delta^{(j)}(x) L_{J,k,j} = \sum_{\ell=0}^{J+2} M_{J,k,\ell}(x) \left(P_{0}^{[\ell]}(x)y_0 + P_{1}^{[\ell]}(x)y_1 + P_{2}^{[\ell]}(x)y_2 \right),
\end{align*}
where now the $L_{J,k,j}$ are linear forms in the independent variables $y_0,y_1,y_2$. Let $\alpha$ be a nonzero complex number with $T(\alpha)\neq 0$, and let $a = \text{ord}_{x=\alpha}\Delta$. Taking $J=a$ yields
\begin{align*}
T(\alpha)^a\Delta^{(a)}(\alpha) y_k = \sum_{\ell=0}^{a+2} M_{a,k,\ell}(\alpha) \left(P_{0}^{[\ell]}(\alpha)y_0 + P_{1}^{[\ell]}(\alpha)y_1 + P_{2}^{[\ell]}(\alpha)y_2 \right).
\end{align*}
By assumption we have $T(\alpha)\neq 0$ and $\Delta^{(a)}(\alpha) \neq 0$, and therefore each variable $y_k$ is a linear combination of the $a+3$ linear forms $P_{0}^{[\ell]}(\alpha)y_0 + P_{1}^{[\ell]}(\alpha)y_1 + P_{2}^{[\ell]}(\alpha)y_2$.
\end{proof}

\section{Alternate Pad\'e polynomials}\label{sec:alternate pade polys}

Lemma \ref{lem:full rank for square bracket polys} shows the polynomials $P_i^{[k]}$ are suitably independent, but it is cumbersome to bound the size of $P_i^{[k]}(\alpha), \alpha \in \mathbb{C}$. In this section we introduce the polynomials $P_i^{\langle k \rangle}$, also derived from the initial Pad\'e polynomials of Lemma \ref{lem:initial pade polynomials}, for which it is easier to control the size of $P_i^{\langle k \rangle}(\alpha)$. The polynomials $P_i^{\langle k \rangle}$ and $P_i^{[ k ]}$ are related by a nonsingular transformation, so the new polynomials $P_i^{\langle k \rangle}$ inherit the independence of the polynomials $P_i^{[ k ]}$. The rational numbers $p_{i,j}$ of Proposition \ref{prop:existence of good polynomials} arise from evaluating the polynomials $P_i^{\langle k \rangle}(x)$ at $x=1/n$ for some suitably chosen $n$ (see Lemma \ref{lem:full rank for angle bracket polys} and Lemma \ref{lem:rationals coming from Pade polynomials}).

Let $P_i, 0\leq i \leq 2$, and $R$ be as in Lemma \ref{lem:initial pade polynomials}. We recall the definitions \eqref{eq:defn of A funcs} and \eqref{eq:defn of T poly} of $A_i$ and $T$, respectively, and define for $k\geq 0$
\begin{align}\label{eq:defn of R and P langle rangle}
\begin{split}
P_i^{\langle k \rangle}(x) &:= \frac{T(x)^k}{k!}\left(\frac{d}{dx} + A_i(x) \right)^k P_i(x), \\
R^{\langle k \rangle}(x) &:= \frac{T(x)^k}{k!}R^{(k)}(x).
\end{split}
\end{align}
The following lemma relates the polynomials $P_i^{\langle k \rangle}$ and $P_i^{[k]}$.

\begin{lemma}[Relationship between Pad\'e polynomials]\label{lem:relating the two kinds of Pade polynomials}
Let $A \in \mathbb{Q}(x)$, and let $P,T \in \mathbb{Z}[x]$. Then for $k\geq 0$ there exist polynomials $q_{k,j}\in \mathbb{Z}[x]$ with
\begin{align*}
\left( T \left( \frac{d}{dx} + A\right) \right)^k P = T^k\left( \frac{d}{dx} + A \right)^k P + \sum_{j=1}^{k-1}q_{k,j} T^j \left( \frac{d}{dx} + A\right)^j P
\end{align*}
and $\textup{deg}(q_{k,j}) \leq (k-j)\textup{deg}(T)$. The polynomials $q_{k,j}$ depend on $T$ but are independent of $A$ and $P$.
\end{lemma}
\begin{proof}
This is \cite[Lemma 1.11]{Nag1995}. We proceed by induction on $k$. The claim is trivial for $k=0$ and $k=1$. Assume that
\begin{align*}
\left( T \left( \frac{d}{dx} + A\right) \right)^k P =q_{k,k} T^k\left( \frac{d}{dx} + A \right)^k P + \sum_{j=1}^{k-1}q_{k,j} T^j \left( \frac{d}{dx} + A\right)^j P,
\end{align*}
where $q_{k,k}=1$. Apply $T(\frac{d}{dx}+A)$ to both sides to obtain
\begin{align*}
\left( T \left( \frac{d}{dx} + A\right) \right)^{k+1} P &= T\bigg(\frac{d}{dx}+A\bigg)\bigg(T^k\left( \frac{d}{dx} + A \right)^k P \bigg) \\
&\qquad\qquad+ \sum_{j=1}^{k-1}T\bigg(\frac{d}{dx}+A\bigg)\bigg(q_{k,j} T^j \left( \frac{d}{dx} + A\right)^j P \bigg).
\end{align*}
We derive
\begin{align*}
T\bigg(\frac{d}{dx}+A\bigg)\bigg(T^k\left( \frac{d}{dx} + A \right)^k P \bigg) &= k T^k T' \left( \frac{d}{dx}+A\right)^k P + T^{k+1}\left( \frac{d}{dx}+A\right)^{k+1} P
\end{align*}
and
\begin{align*}
T\bigg(\frac{d}{dx}+A\bigg)\bigg(q_{k,j} T^j \left( \frac{d}{dx} + A\right)^j P \bigg) &= q_{k,j}'T^{j+1} \left( \frac{d}{dx} + A\right)^j P + jq_{k,j} T^j T' \left( \frac{d}{dx} + A\right)^j P \\ 
&\qquad\qquad+ q_{k,j}T^{j+1} \left( \frac{d}{dx} + A\right)^{j+1} P.
\end{align*}
If we define
\begin{align}\label{eq:q polynomials}
\begin{split}
q_{k+1,k+1} &= 1, \\
q_{k+1,1} &= \frac{d}{dx} (q_{k,1} T), \\
q_{k+1,j} &= q_{k,j}' T + jq_{k,j} T'+q_{k,j-1}, \ \ \ \ \ \  \ 2\leq j \leq k,
\end{split}
\end{align}
then we obtain
\begin{align*}
\left( T \left( \frac{d}{dx} + A\right) \right)^{k+1} P =q_{k+1,k+1} T^{k+1}\left( \frac{d}{dx} + A \right)^{k+1} P + \sum_{j=1}^{k}q_{k+1,j} T^j \left( \frac{d}{dx} + A\right)^j P,
\end{align*}
which completes the induction. The properties claimed for the polynomials $q_{k,j}$ follow from \eqref{eq:q polynomials} by another induction.
\end{proof}

We have the following analogue of Lemma \ref{lem:props of square bracket Pade polynomials} (see also \cite[Lemma 1.12]{Nag1995}).

\begin{lemma}[Properties of $P_i^{\langle k \rangle}$ and $R^{\langle k \rangle}$]\label{lem:props of angle bracket Pade polynomials}
Let $P_i^{\langle k \rangle}$ and $R^{\langle k \rangle}$ be defined as in \eqref{eq:defn of R and P langle rangle}. The following are true for $k\geq 0$\textup{:}
\begin{enumerate}
\item $P_i^{\langle k \rangle}(x) \in \mathbb{Q}[x]$,
\item $\textup{deg}(P_i^{\langle k \rangle}) \leq D+k\mathcal{M}$,
\item $R^{\langle k \rangle}(x) = \sum_{i=0}^2 P_i^{\langle k \rangle}(x) \omega_i(x)$.
\end{enumerate}
\end{lemma}
\begin{proof}
It is easy to prove by induction that
\begin{align*}
\frac{d^k}{dx^k}(P_i(x) \omega_i(x)) = \omega_i \cdot \left(\frac{d}{dx} + A_i \right)^k P_i,
\end{align*}
which gives the third claim. For the other two claims we use the identity
\begin{align}\label{eq:relate square bracket and angle bracket}
k! P_i^{\langle k \rangle}(x) &= P_i^{[k]}(x) - \sum_{j=1}^{k-1} q_{k,j}(x) j! P_i^{\langle j \rangle}(x),
\end{align}
which follows from Lemma \ref{lem:relating the two kinds of Pade polynomials}. Since $P_i^{[k]},q_{k,j}\in \mathbb{Z}[x]$, an induction shows $P_i^{\langle k \rangle} \in \mathbb{Q}[x]$.

It remains to prove that $\textup{deg}(P_i^{\langle k \rangle}) \leq D+k\mathcal{M}$, which we also do by induction. The claim holds true for $k=0,1$. Now assume that the claim is true for all $j < k$. Lemma \ref{lem:props of square bracket Pade polynomials} gives $\text{deg}(P_i^{[k]})\leq D+k\mathcal{M}$, and by Lemma \ref{lem:relating the two kinds of Pade polynomials} and the induction hypothesis we have $\text{deg}(q_{k,j}P_i^{\langle j \rangle})\leq (k-j)\mathcal{M} + D + j\mathcal{M} = D+k\mathcal{M}$. Since $\text{deg}(P_1+P_2)\leq \max(\text{deg}(P_1),\text{deg}(P_2))$ the relation \eqref{eq:relate square bracket and angle bracket} completes the induction.
\end{proof}

\begin{lemma}[$P_i^{\langle k \rangle}$ matrix has full rank]\label{lem:full rank for angle bracket polys}
Let $P_i^{\langle k \rangle}$ be defined as in \eqref{eq:defn of P square bracket} and let $\Delta$ be defined as in \eqref{eq:defn of det poly}. Let $\alpha \in \mathbb{C}\backslash\{0\}$ be such that $T(\alpha)\neq 0$. Assume $D > \frac{\mathcal{M} + 2}{1-\epsilon_0}$. Write $a=\textup{ord}_{x=\alpha}\Delta$ ($a$ is a nonnegative integer). Then the matrix
\begin{align*}
(P_i^{\langle k \rangle}(\alpha))_{\substack{0\leq k \leq a+2 \\ 0\leq i \leq 2}}
\end{align*}
has rank three.
\end{lemma}
\begin{proof}
The argument is that of \cite[Lemma 1.13]{Nag1995}. For a nonnegative integer $L$, Lemma \ref{lem:relating the two kinds of Pade polynomials} implies
\begin{align*}
\begin{pmatrix}
P_0^{[0]}(x) & P_1^{[0]}(x) & P_2^{[0]}(x) \\
&\vdots \\
P_0^{[L]}(x) & P_1^{[L]}(x) & P_2^{[L]}(x)
\end{pmatrix} &= M (x)
\begin{pmatrix}
P_0^{\langle 0 \rangle}(x) & P_1^{\langle 0 \rangle}(x) & P_2^{\langle 0 \rangle}(x) \\
&\vdots \\
P_0^{\langle L \rangle}(x) & P_1^{\langle L \rangle}(x) & P_2^{\langle L \rangle}(x)
\end{pmatrix},
\end{align*}
where $M(x)$ is the $(L+1) \times (L+1)$ lower triangular matrix
\begin{align*}
M(x) &= 
\begin{pmatrix}
1 & 0 & 0 & 0 & \cdots & 0 & 0 \\
0 & 1 & 0 & 0 & \cdots & 0 & 0 \\
0 & q_{2,1}(x) & 2! & 0 & \cdots & 0 & 0 \\
0 & q_{3,1}(x) & 2! q_{3,2}(x) & 3! & \cdots & 0 & 0 \\
&\vdots & & & \ddots \\
0 & q_{L,1}(x) & 2! q_{L,2}(x) & 3! q_{L,3}(x) & \cdots & (L-1)! q_{L,L-1}(x) & L!
\end{pmatrix}.
\end{align*}
The matrix $M(\alpha)$ is nonsingular for any $\alpha \in \mathbb{C}$ since $\det (M(\alpha)) = \prod_{j=0}^L j! \neq 0$,
so the result follows from Lemma \ref{lem:full rank for square bracket polys}.
\end{proof}

We are now in a position to study some properties of the rational numbers $P_i^{\langle k \rangle} (1/n)$.

\begin{lemma}[Rational numbers from Pad\'e polynomials]\label{lem:rationals coming from Pade polynomials}
Let $P_i^{\langle k \rangle}$ be defined as in \eqref{eq:defn of P square bracket}, and let $n \in [N,2N)$. The following are true:
\begin{enumerate}
\item $P_i^{\langle k \rangle}$ is a polynomial with rational coefficients whose denominators divide $r^k$.
\item $P_i^{\langle k \rangle}(1/n)$ is a rational number (possibly equal to zero) with denominator dividing $r^k n^{D+k\mathcal{M}}$.
\item $|P_i^{\langle k \rangle}(1/n)| \leq (2rN)^k 10^{\mathcal{M}\epsilon_0^{-1}}(4r^4\mathcal{M})^{(2-\epsilon_0)(3-\epsilon_0)\epsilon_0^{-1}D}$.
\end{enumerate}
\end{lemma}
\begin{proof}
The argument is essentially that of \cite[Lemma 2.6]{Nag1995}. We recall from Lemma \ref{lem:diff eq for omega funcs} and \eqref{eq:defn of A funcs} that $\omega_i'(x) = A_i(x) \omega(x)$. Now define $W_i = P_i \omega_i$, and recall from the proof of Lemma \ref{lem:props of angle bracket Pade polynomials} that
\begin{align*}
W_i^{(k)} &= \omega_i \cdot \left( \frac{d}{dx} + A_i\right)^k P_i.
\end{align*}
On the other hand, we have
\begin{align*}
\frac{1}{k!}W_i^{(k)} &= \frac{1}{k!} \sum_{\ell=0}^k {k \choose \ell} P_i^{(k-\ell)} \omega_i^{(\ell)},
\end{align*}
and it follows that
\begin{align}\label{eq:expanded form of P angel bracket}
P_i^{\langle k \rangle}(x) &= \sum_{\ell=0}^k \frac{T(x)^{k-\ell}}{(k-\ell)!} P_i^{(k-\ell)}(x) \cdot \frac{T(x)^\ell}{\ell!}\frac{\omega_i^{(\ell)}(x)}{\omega_i(x)}.
\end{align}
We have $\frac{1}{(k-\ell)!} P_i^{(k-\ell)} \in \mathbb{Z}[x]$ since $P_i \in \mathbb{Z}[x]$ and $\frac{1}{h!}\frac{d^h}{dx^h}(x^n) \in \mathbb{Z}[x]$ for any $h,n\geq 0$. Now we turn to $\frac{T(x)^\ell}{\ell!}\frac{\omega_i^{(\ell)}(x)}{\omega_i(x)}$. Taking derivatives yields
\begin{align*}
\omega_i^{(\ell)}(x) &= \sum_{\substack{k_1 + \cdots + k_{\beta_i-1} = \ell \\ k_j \geq 0}} \frac{\ell!}{k_1! \cdots k_{\beta_i-1}!}\prod_{j=1}^{\beta_i-1} \frac{d^{k_j}}{dx^{k_j}}((1-jx)^{-1/r}) \\
&= \sum_{\substack{k_1 + \cdots + k_{\beta_i-1} = \ell \\ k_j \geq 0}} \ell! \prod_{j=1}^{\beta_i-1} j^{k_j}(-1)^{k_j} {{-1/r}\choose k_j} (1-jx)^{-1/r-k_j},
\end{align*}
and by Lemma \ref{lem:denom of binomial coeff}
\begin{align*}
r^\ell\frac{T(x)^\ell}{\ell!}\frac{\omega_i^{(\ell)}(x)}{\omega_i(x)} &= \sum_{\substack{k_1 + \cdots + k_{\beta_i-1} = \ell \\ k_j \geq 0}} r^{2\ell}\prod_{j=1}^{\beta_i-1} j^{k_j}(-1)^{k_j} {{-1/r}\choose k_j} (1-jx)^{\ell-k_j}\prod_{j=\beta_i}^{\mathcal{M}} (1-jx)^\ell
\end{align*}
is a polynomial with integer coefficients. This proves the first claim, upon recalling that $\ell \leq k$.

For the second claim, note that $P_i^{\langle k \rangle}$ is a polynomial of degree $\leq D + k\mathcal{M}$ (Lemma \ref{lem:props of angle bracket Pade polynomials}) with coefficients which are rational numbers whose denominators divide $r^k$. Then $P_i^{\langle k \rangle}(1/n)$ is a rational number with denominator dividing $r^k n^{D+k\mathcal{M}}$.

It remains to prove the third claim of the lemma. By \eqref{eq:expanded form of P angel bracket} and the triangle inequality
\begin{align*}
|P_i^{\langle k \rangle}(1/n)| &\leq \sum_{\ell=0}^k \frac{T(1/n)^{k-\ell}}{(k-\ell)!} |P_i^{(k-\ell)}(1/n)| \cdot \frac{T(1/n)^\ell}{\ell!}\frac{|\omega_i^{(\ell)}(1/n)|}{|\omega_i(1/n)|}.
\end{align*}
We have $T(1/n)^{k-\ell} = r^{k-\ell} \prod_{j=1}^{\mathcal{M}}(1-\tfrac{j}{n})^{k-\ell} < r^{k-\ell}$, and, writing $p_{i,v}$ for the coefficients of $P_i$,
\begin{align*}
\frac{1}{(k-\ell)!}P_i^{(k-\ell)}(x) &= \sum_{v=k-\ell}^D p_{i,v} \frac{1}{(k-\ell)!} \frac{d^{k-\ell}}{dx^{k-\ell}}(x^v) = x^{-(k-\ell)}\sum_{v=k-\ell}^D p_{i,v} {v \choose {k-\ell}} x^v,
\end{align*}
so by Lemma \ref{lem:initial pade polynomials}
\begin{align*}
\frac{1}{(k-\ell)!} |P_i^{(k-\ell)}(1/n)|&\leq (2N)^{k-\ell}4^{\mathcal{M}\epsilon_0^{-1}}(4r^4\mathcal{M})^{(2-\epsilon_0)(3-\epsilon_0)\epsilon_0^{-1} D} \sum_{v=k-\ell}^D \left(\frac{4\mathcal{M}}{N} \right)^v \\
&\leq (2N)^{k-\ell}5^{\mathcal{M}\epsilon_0^{-1}}(4r^4\mathcal{M})^{(2-\epsilon_0)(3-\epsilon_0)\epsilon_0^{-1} D}.
\end{align*}
Applying the triangle inequality again,
\begin{align*}
\frac{T(1/n)^\ell}{\ell!}\frac{|\omega_i^{(\ell)}(1/n)|}{|\omega_i(1/n)|} &\leq r^\ell \sum_{\substack{k_1 + \cdots + k_{\beta_i-1} = \ell \\ k_j \geq 0}} \prod_{j=1}^{\beta_i-1} j^{k_j} \leq (r\mathcal{M})^\ell \sum_{\substack{k_1 + \cdots + k_{\beta_i-1} = \ell \\ k_j \geq 0}}1 \leq 2^\mathcal{M} (2r\mathcal{M})^\ell,
\end{align*}
and therefore
\begin{align*}
|P_i^{\langle k \rangle}(1/n)| &\leq 2^{\mathcal{M}+1} (2rN)^k 5^{\mathcal{M}\epsilon_0^{-1}}(4r^4\mathcal{M})^{(2-\epsilon_0)(3-\epsilon_0)\epsilon_0^{-1}D} \\
&\leq (2rN)^k 10^{\mathcal{M}\epsilon_0^{-1}}(4r^4\mathcal{M})^{(2-\epsilon_0)(3-\epsilon_0)\epsilon_0^{-1}D}. \qedhere
\end{align*}
\end{proof}

\section{Proof of Proposition \ref{prop:existence of good polynomials}}\label{sec:proof of key prop}

Assume the hypotheses of Proposition \ref{prop:existence of good polynomials}. In particular, we have
\begin{align}\label{eq:many solutions}
\sum_{\substack{\mathcal{N} \leq n < \mathcal{N} + \lfloor \frac{N}{\log N}\rfloor \\ N \leq n < 2N \\ sn! = P(x) \\ s(n-\beta_i)! = P(x_i), i=1,2}} 1 > \frac{N}{\mathcal{M}^3 (\log N)}.
\end{align}
Given $D$ and $\beta_i, \omega_i$ as in \eqref{eq:defn of omega functions}, we apply Lemma \ref{lem:initial pade polynomials} to construct Pad\'e polynomials $P_i(x), 0\leq i \leq 2$. Since $D \asymp_{P,s} N > 10^6 \mathcal{M}$, Lemma \ref{lem:full rank for square bracket polys} implies the determinant polynomial $\Delta(x)$ is not identically zero. Now we avail ourselves of the following trivial lemma.

\begin{lemma}[Polynomial vanishing to low order at a point]\label{lem:point of low order of vanishing}
Let $Q \in \mathbb{C}[x]$ be a nonzero polynomial of degree $\leq d$. Let $x_1,\ldots,x_t$ be distinct complex numbers. Then some $x_i$ satisfies
\begin{align*}
\textup{ord}_{x=x_i}\, Q(x) \leq \frac{d}{t}.
\end{align*}
\end{lemma}
\begin{proof}
A nonzero polynomial of degree $\leq d$ has $\leq d$ roots (counted with multiplicity) so
\begin{align*}
t\cdot \min_{1\leq i \leq t} \text{ord}_{x=x_i} Q(x) &\leq \sum_{i=1}^t\text{ord}_{x=x_i} Q(x) \leq d. \qedhere
\end{align*}
\end{proof}

We have $\Delta \neq 0$ and $\text{deg}(\Delta) \leq 3D + 3 \mathcal{M} \leq 4D$. Lemma \ref{lem:point of low order of vanishing} and \eqref{eq:many solutions} imply there exists some solution $n_0 \in [\mathcal{N},\mathcal{N} +\lfloor \frac{N}{\log N}\rfloor) \cap [N,2N)$ such that
\begin{align*}
a := \text{ord}_{x=\alpha} \Delta(x) &\leq 4D \cdot \frac{\mathcal{M}^3(\log N)}{N} \ll_{P,s} \mathcal{M}^3 (\log N),
\end{align*}
where we have written $\alpha = n_0^{-1}$. Since $T(\alpha) \neq 0$, Lemma \ref{lem:full rank for angle bracket polys} yields integers $0\leq k_0 <k_1 < k_2 \leq a+2$ such that
\begin{align*}
\text{det}
\begin{pmatrix}
P_0^{\langle k_0 \rangle}(\alpha) & P_1^{\langle k_0 \rangle}(\alpha)& P_2^{\langle k_0 \rangle}(\alpha) \\
P_0^{\langle k_1 \rangle}(\alpha) & P_1^{\langle k_1 \rangle}(\alpha) & P_2^{\langle k_1 \rangle} (\alpha)\\
P_0^{\langle k_2 \rangle}(\alpha) & P_1^{\langle k_2 \rangle}(\alpha) & P_2^{\langle k_2 \rangle}(\alpha)
\end{pmatrix} \neq 0.
\end{align*}
We therefore set 
\begin{align*}
p_{i,j} := P_{i}^{\langle k_j \rangle}(\alpha) \in \mathbb{Q}.
\end{align*}
If we define $Z := r^{a+2} n_0^{D+(a+2)\mathcal{M}}$ then $Zp_{i,j} \in \mathbb{Z}$ for $0\leq i,j\leq 2$ by Lemma \ref{lem:rationals coming from Pade polynomials}. We have $Z\leq (2r\mathcal{N})^{\mathcal{M}^5} (2\mathcal{N})^D$ and
\begin{align*}
|p_{i,j}| &\leq (4r\mathcal{N})^{\mathcal{M}^4} 10^{\mathcal{M}\epsilon_0^{-1}}\left\{(4r^4\mathcal{M})^{(2-\epsilon_0)(3-\epsilon_0)\epsilon_0^{-1}}\right\}^D,
\end{align*}
also by Lemma \ref{lem:rationals coming from Pade polynomials}.

By the definition of $R^{\langle k \rangle}$ and Lemma \ref{lem:initial pade polynomials} we have
\begin{align*}
\sum_{i=0}^2 p_{i,j}\omega_i(\alpha) &= R^{\langle k_j \rangle} (\alpha) = T(\alpha)^{k_j} \sum_{v\geq \mathcal{O}} r_v \frac{1}{k_j!} \frac{d^{k_j}}{dx^{k_j}}(x^v)\Big|_{x=\alpha} = T(\alpha)^{k_j} \sum_{v\geq \mathcal{O}} r_v {v \choose {k_j}} \alpha^{v-k_j},
\end{align*}
and therefore
\begin{align*}
\left|\sum_{i=0}^2 p_{i,j}\omega_i(\alpha) \right| &\leq (2r\mathcal{N})^{a+2}\sum_{v\geq \mathcal{O}} |r_v| (2\alpha)^v \\ 
&\leq (2r\mathcal{N})^{a+2} 4^{\mathcal{M}\epsilon_0^{-1}}(4r^4 \mathcal{M})^{(2-\epsilon_0)(3-\epsilon_0)\epsilon_0^{-1} D} \sum_{v\geq \mathcal{O}} (4\mathcal{M}\alpha)^v \\
&\leq \frac{1}{1 - \frac{4\mathcal{M}}{\mathcal{N}}} (2r\mathcal{N})^{a+2} 4^{\mathcal{M}\epsilon_0^{-1}}(4r^4 \mathcal{M})^{(2-\epsilon_0)(3-\epsilon_0)\epsilon_0^{-1} D} \left(\frac{4\mathcal{M}}{\mathcal{N}}\right)^{\mathcal{O}} \\
&\leq (2r\mathcal{N})^{a+3}4^{\mathcal{M}\epsilon_0^{-1}} \left(\frac{(16r^4\mathcal{M})^{(2-\epsilon_0)(3-\epsilon_0)\epsilon_0^{-1} + 3-\epsilon_0}}{\mathcal{N}^{\, 3-\epsilon_0}} \right)^D \\
&\leq (2r\mathcal{N})^{2\mathcal{M}^4} 4^{\mathcal{M}\epsilon_0^{-1}} \left(\frac{{\mathcal{N}}^{\, 3-\epsilon_0}}{(16r^4 \mathcal{M})^{2(3-\epsilon_0)\epsilon_0^{-1}}} \right)^{-D}.
\end{align*}
This completes the proof of Proposition \ref{prop:existence of good polynomials}.

\section{Possible extensions and investigations}\label{sec:future work}

We close the paper by offering additional questions one might study.

\begin{enumerate}
\item The most obvious question is whether the exponent $\frac{33}{34}$ of Theorem \ref{thm:main thm} can be reduced. One possibility is to determine whether the polynomials $P_i$ of Lemma \ref{lem:initial pade polynomials} can be made explicit, as can be done in some situations when the functions $\omega_i(x)$ are replaced by binomial functions $(1+x)^\alpha$ \cite{Ric1993, Ben2001}. Another possibility is to make an in-depth study of the coefficients of the $\omega_i$ functions and attempt to gain from more advanced forms of Siegel's lemma \cite{BV1983}.

\item The upper bound of Theorem \ref{thm:main thm} depends on the polynomial $P$. It would be desirable to prove an upper bound of the form $CN^{1-\delta}$ where the constant $C$ depends only on the degree of the polynomial.

\item It is an easy exercise to show that, for any degree $r\geq 2$, there are infinitely many integer polynomials of degree $r$ representing $r+1$ factorial values. It would be interesting to know whether there are polynomials of degree $r$ representing $\geq r+2$ factorial values. This is a generalization of a question of Ulas \cite[Question 2.6]{Ulas2012}.

One conjectures that for any $r\geq 2$ there exists a positive constant $C_r$ such that, for any degree $r$ polynomial $P$, the equation $n! = P(x)$ has $\leq C_r$ solutions. Such a result, if true, must be very deep, since it does not seem to follow from the ABC conjecture. However, one might ask whether the result follows from some combination of standard conjectures in arithmetic geometry (ABC, Bombieri-Lang, Vojta, etc).

\item It would be interesting to obtain bounds, along the lines of Theorem \ref{thm:main thm}, for the number of solutions to the equation $H_n = P(x)$, where $H_n$ is a ``highly divisible'' sequence as in \cite{BH2006}. For example, one might consider $H_n = \text{lcm}(1,2,\ldots,n)$, or $H_n =p_1p_2\cdots p_n$, the product of the first $n$ primes. In both of these cases it seems difficult to obtain a power saving bound as in Theorem \ref{thm:main thm}. However, one should be able to obtain some effective saving over the trivial bound by adapting our techniques.

\item Let $k\geq 2$ be a positive integer, and let $\delta>0$ be a small positive constant. The ABC conjecture implies there are only finitely many solutions (depending on $k$ and $\delta$, with $n$ and $x$ coprime) to $n! = x^k + O(x^{k-1-\delta})$. On the other hand, an easy greedy argument gives infinitely many solutions to $|n! - x^k| \ll_k x^{k-1}$.

It is possible to adapt our method to obtain a bound of the form $\ll_{k,\delta} N^{1 - \varepsilon(\delta)}$ for the number of solutions to $n! = x^k + O(x^{k-1-\delta})$, where $\varepsilon(\delta) \rightarrow 0$ as $\delta \rightarrow 0$. One works with $n\geq n_0(\delta)$ solutions rather than three solutions as in Lemma \ref{lem:solutions imply simul approx}, the key point being that if $\mathcal{M}$ is a very small power of $N$ then the analogous quantity to $\chi$ in the proof of Proposition \ref{prop:not many close solutions} satisfies $\chi \approx \frac{1}{n}$. We leave the details to the interested reader.
\end{enumerate}

\section*{Acknowledgements}
We thank Daniel Berend for making us aware of the reference \cite{Sau2020}, and Florian Luca for drawing our attention to some typos in an earlier version of the paper.

\bibliographystyle{amsalpha}

\end{document}